\RequirePackage{fix-cm}
\documentclass[smallextended,referee,envcountsect]{svjour3}       
\smartqed  

\usepackage{graphicx}
\usepackage{amsmath, ntheorem,dsfont}

\usepackage{pifont}

\usepackage{amssymb}
\usepackage{stackrel}
\usepackage{marvosym, enumerate}
\usepackage{amstext, amscd, latexsym}
\usepackage{amsfonts, ragged2e}
\usepackage{mathrsfs, pgfplots, enumerate, setspace}
\usepackage{subfigure}
\usepackage{cases}

\smartqed
\usepackage{cite}
\usepackage{amstext, graphicx, amscd, latexsym, amssymb}
\usepackage{amsfonts, ragged2e}
\usepackage{pgfplots, setspace}
\usepackage{geometry}
\usepackage{latexsym,bm}
\usepackage{float}
\usepackage{epstopdf,caption}
 
\usepackage{amsfonts,amsmath,amsbsy,amssymb,amscd,mathrsfs}
\usepackage{graphicx}
\usepackage{epstopdf}
\usepackage{algorithmic} 
\usepackage{graphicx,epstopdf}

 \newcommand{\nm}[1]{\left\lVert {#1} \right\rVert}
 
 \newcommand{\dual}[1]{\left\langle {#1} \right\rangle}

\usepackage[figuresright]{rotating}
\usepackage[sort&compress,numbers]{natbib}
\usepackage[ruled,linesnumbered]{algorithm2e} 
\usepackage[colorlinks,linkcolor=blue,anchorcolor=blue,citecolor=blue,urlcolor=blue]{hyperref}

\newtheorem{assumption}{Assumption}[section]

\journalname{}

\begin{document}

\title{A Subspace Minimization Barzilai-Borwein Method for Multiobjective Optimization Problems}


\author{Jian Chen \and Liping Tang \and  Xinmin Yang  }

\institute{J. Chen \at National  Center  for  Applied  Mathematics in Chongqing, Chongqing Normal University, Chongqing 401331, China, and School of Mathematical Sciences, University of Electronic Science and Technology of China, Chengdu, Sichuan 611731, China\\
                    \href{mailto:chenjian_math@163.com}{chenjian\_math@163.com}\\
                   L.P. Tang \at National Center for Applied Mathematics in Chongqing, Chongqing Normal University, Chongqing 401331, China\\
                   \href{mailto:tanglipings@163.com}{tanglipings@163.com}\\
        \Letter X.M. Yang \at National Center for Applied Mathematics in Chongqing, and School of Mathematical Sciences,  Chongqing Normal University, Chongqing 401331, China\\
        \href{mailto:xmyang@cqnu.edu.cn}{xmyang@cqnu.edu.cn} \\}

\date{Received: date / Accepted: date}

\maketitle

\begin{abstract}
Nonlinear conjugate gradient methods have recently garnered significant attention within the multiobjective optimization community. These methods aim to maintain consistency in conjugate parameters with their single-objective optimization counterparts. However, the preservation of the attractive conjugate property of search directions remains uncertain, even for quadratic cases, in multiobjective conjugate gradient methods. This loss of interpretability of the last search direction significantly limits the applicability of these methods. To shed light on the role of the last search direction, we introduce a novel approach called the subspace minimization Barzilai-Borwein method for multiobjective optimization problems (SMBBMO). In SMBBMO, each search direction is derived by optimizing a preconditioned Barzilai-Borwein subproblem within a two-dimensional subspace generated by the last search direction and the current Barzilai-Borwein descent direction. Furthermore, to ensure the global convergence of SMBBMO, we employ a modified Cholesky factorization on a transformed scale matrix, capturing the local curvature information of the problem within the two-dimensional subspace. Under mild assumptions, we establish both global and $Q$-linear convergence of the proposed method. Finally, comparative numerical experiments confirm the efficacy of SMBBMO, even when tackling large-scale and ill-conditioned problems.

\keywords{Multiobjective optimization \and Subspace method \and Barzilai-Borwein's method \and Global convergence}
\subclass{90C29 \and 90C30}
\end{abstract}

\section{Introduction}
An unconstrained multiobjective optimization problem (MOP) is typically formulated as follows:
\begin{align*}
	\min\limits_{x\in\mathbb{R}^{n}} F(x), \tag{MOP}\label{MOP}
\end{align*}
where $F:\mathbb{R}^{n}\rightarrow\mathbb{R}^{m}$ is a continuously differentiable function. This type of problem finds widespread applications across various domains, including engineering \cite{MA2004}, economics \cite{FW2014}, management science \cite{E1984}, and machine learning \cite{SK2018}, among others. These applications often involve the simultaneous optimization of multiple objectives. However, achieving a single solution that optimizes all objectives is often impractical. Therefore, optimality is defined by {\it Pareto optimality} or {\it efficiency}. A solution is deemed Pareto optimal or efficient if no objective can be improved without sacrificing the others. 
\par  In the past two decades, multiobjective gradient descent methods have gained significant traction within the multiobjective optimization community. These methods determine descent directions by solving subproblems, followed by the application of line search techniques along these directions to ensure sufficient improvement across all objectives. The origins of multiobjective gradient descent methods can be traced back to pioneering works by Mukai \cite{M1980} and Fliege and Svaiter \cite{FS2000}. The latter clarified that the multiobjective steepest descent direction reduces to the steepest descent direction when dealing with a single objective. This observation inspired researchers to extend ordinary numerical algorithms for solving MOPs (see, e.g., \cite{AP2021,BI2005,CL2016,FD2009,FV2016,GI2004,LP2018,MP2019,P2014,QG2011} and references therein). 
\par Recently,  Lucambio P\'{e}rez and Prudente \cite{LP2018} made significant advancements by extending Wolfe line search and the Zoutendijk condition to multiobjective optimization, thereby facilitating the exploration of multiobjective nonlinear conjugate gradient methods. These methods leverage both the current steepest descent direction and the last search direction to construct the current search direction, ensuring consistency in conjugate parameters with their counterparts in single-objective optimization problems (SOPs), such as Fletcher--Reeves \cite{LP2018}, Conjugate descent \cite{LP2018}, Dai--Yuan \cite{LP2018},  Polak--Ribi\`{e}re--Polyak \cite{LP2018}, Hestenes--Stiefel \cite{LP2018}, Hager--Zhang \cite{GP2020,HZC2024} and Liu--Storey \cite{GLP2022}.
\par The linear conjugate gradient method exhibits finite termination for convex quadratic minimization, owing to its attractive conjugate property.  
In multiobjective optimization, Fukuda et al. \cite{FGM2022} proposed a conjugate directions-type that achieves finite termination for strongly convex quadratic MOPs. Unfortunately, the method cannot be extended to non-quadratic cases. Moreover, the attractive conjugate property of search directions remains unknown for quadratic cases in existing multiobjective conjugate gradient methods. Conversely, in single-objective optimization, the absence of the conjugate property severely constrains the application of nonlinear conjugate gradient methods, particularly in large-scale non-quadratic cases. To tackle this challenge, Yuan and Stoer \cite{YS1995} devised the subspace minimization conjugate gradient (SMCG) method for SOPs. The search directions of SMCG are obtained by optimizing approximate models within two-dimensional subspaces generated by gradient and last search directions. Consequently, the conjugate parameters of SMCG is optimal with respect to approximate models. An advantage of SMCG is that lower-dimensional subspaces enable us to solve the corresponding subproblems efficiently. Furthermore, in many cases, the subspace approaches achieve comparable theoretical properties to their full-space counterparts. As described above, the extension of SMCG to MOPs is of great interest. Naturally, the key issues for such a method are how to choose the subspaces and how to obtain the approximate models for better curvature exploration.
\par In this paper, we propose a subspace minimization Barzilai-Borwein method for MOPs (SMBBMO), the choice of subspaces and approximate models are described as follows:
\par$\bullet$ {\it Subspace}: Chen et al. \cite{CTY2023a,CTY2023b} highlighted that imbalances among objectives seriously decelerate the convergence of steepest descent method. To alleviate the impact of imbalances, we propose constructing a subspace using both the current Barzilai-Borwein descent direction \cite{CTY2023a} and the previous search direction. This construction is defined as follows:
\begin{equation*}
	d_{k}=\left\{
	\begin{aligned}
		\begin{split}
			&v_{k}, & &{\rm if}~k=0, \\
			&\mu_{k}v_{k}+\nu_{k}d_{k-1}, & &{\rm if}~k\geq1,
		\end{split}
	\end{aligned}
	\right.
\end{equation*} 
where $(\mu_{k},\nu_{k})^{T}\in\mathbb{R}^{2}$, $v_{k}$ is the Barzilai-Borwein descent direction at $x^{k}$. Notably, $d^{k}$ follows a formula similar to the spectral conjugate gradient direction due to the Barzilai-Borwein descent direction $v_{k}$. However, existing multiobjective spectral conjugate gradient methods \cite{HCL2023,EE2024} confine their search directions to the subspace generated by the current steepest descent direction and the previous search direction. Consequently, these approaches may not fully leverage the benefits of spectral information for each objective.
\par$\bullet$ \textit{Approximate model}: To strike a balance between per-iteration cost and improved curvature exploration, we adopt the preconditioned Barzilai-Borwein subproblem \cite{CTY2024p} as the initial approximate model. To circumvent the need for matrix calculations, we utilize finite differences of gradients to estimate matrix-vector products. Consequently, the local curvature information of the problem in the two-dimensional subspace is represented by a $2\times2$ matrix. Similar to SMCG (Li et al., 2024), the global convergence of SMBBMO in non-convex cases remains uncertain when employing the $2\times2$ matrix selection strategy of SMCG. Motivated by the work of Lapucci et al. \cite{LLL2024}, we apply a modified Cholesky factorization to a transformed scale matrix. This enables us to establish global convergence for the proposed method.
\par The primary objective of SMBBMO is to achieve a faster convergence rate compared to the Barzilai-Borwein descent method while maintaining lower computational costs than the preconditioned Barzilai-Borwein method.
\par The paper is organized as follows. Section \ref{sec2} introduces necessary notations and definitions to be utilized later. In Section \ref{sec3}, we propose SMBBMO and explore the choice of subspace and approximate model. The global convergence and $Q$-linear convergence of SMBBMO are established in Section \ref{sec4}. Section \ref{sec5} presents numerical results demonstrating the efficiency of SMBBMO. Finally, conclusions are drawn at the end of the paper.

\section{Preliminaries}\label{sec2}
Throughout this paper, the $n$-dimensional Euclidean space $\mathbb{R}^{n}$ is equipped with the inner product $\langle\cdot,\cdot\rangle$ and the induced norm $\|\cdot\|$. Denote $\mathbb{S}^{n}_{++}(\mathbb{S}^{n}_{+})$ the set of symmetric (semi-)positive definite matrices in $\mathbb{R}^{n\times n}$.  We denote by $JF(x)\in\mathbb{R}^{m\times n}$ the Jacobian matrix of $F$ at $x$, by $\nabla F_{i}(x)\in\mathbb{R}^{n}$ the gradient of $F_{i}$ at $x$ and by $\nabla^{2}F_{i}(x)\in\mathbb{R}^{n\times n}$ the Hessian matrix of $F_{i}$ at $x$. For a positive definite matrix $H$, the notation $\|x\|_{H}=\sqrt{\langle x,Hx \rangle}$ is used to represent the norm induced by $H$ on vector $x$. For simplicity, we denote $[m]:=\{1,2,...,m\}$, and $$\Delta_{m}:=\left\{\lambda:\sum\limits_{i\in[m]}\lambda_{i}=1,\lambda_{i}\geq0,\ i\in[m]\right\}$$ the $m$-dimensional unit simplex. To prevent any ambiguity, we establish the order $\preceq(\prec)$ in $\mathbb{R}^{m}$ as follows: $$u\preceq(\prec)v~\Leftrightarrow~v-u\in\mathbb{R}^{m}_{+}(\mathbb{R}^{m}_{++}),$$
and in $\mathbb{S}^{n}$ as follows:
$$U\preceq(\prec)V~\Leftrightarrow~V-U\in\mathbb{S}^{n}_{+}(\mathbb{S}^{n}_{++}).$$
\par In the following, we introduce the concepts of optimality for (\ref{MOP}) in the Pareto sense. 
\vspace{2mm}
\begin{definition}\label{def1}
	A vector $x^{\ast}\in\mathbb{R}^{n}$ is called Pareto solution to (\ref{MOP}), if there exists no $x\in\mathbb{R}^{n}$ such that $F(x)\preceq F(x^{\ast})$ and $F(x)\neq F(x^{\ast})$.
\end{definition}
\vspace{2mm}
\begin{definition}\label{def2}
	A vector $x^{\ast}\in\mathbb{R}^{n}$ is called weakly Pareto solution to (\ref{MOP}), if there exists no $x\in\mathbb{R}^{n}$ such that $F(x)\prec F(x^{\ast})$.
\end{definition}
\vspace{2mm}
\begin{definition}\label{def3}
	A vector $x^{\ast}\in\mathbb{R}^{n}$ is called  Pareto critical point of (\ref{MOP}), if
	$$\mathrm{range}(JF(x^{*}))\cap-\mathbb{R}_{++}^{m}=\emptyset,$$
	where $\mathrm{range}(JF(x^{*}))$ denotes the range of linear mapping given by the matrix $JF(x^{*})$.

\end{definition}

\par From Definitions \ref{def1} and \ref{def2}, it is evident that Pareto solutions are always weakly Pareto solutions. The following lemma shows the relationships among the three concepts of Pareto optimality.
\vspace{2mm}
\begin{lemma}[See Theorem 3.1 of \cite{FD2009}] The following statements hold.
	\begin{itemize}
		\item[$\mathrm{(i)}$]  If $x\in\mathbb{R}^{n}$ is a weakly Pareto solution to (\ref{MOP}), then $x$ is Pareto critical point.
		\item[$\mathrm{(ii)}$] Let every component $F_{i}$ of $F$ be convex. If $x\in\mathbb{R}^{n}$ is a Pareto critical point of (\ref{MOP}), then $x$ is weakly Pareto solution.
		\item[$\mathrm{(iii)}$] Let every component $F_{i}$ of $F$ be strictly convex. If $x\in\mathbb{R}^{n}$ is a Pareto critical point of (\ref{MOP}), then $x$ is Pareto solution.
	\end{itemize}
\end{lemma}

\section{Subspace minimization Barzilai-Borwein descent method for MOPs}\label{sec3}
Let us consider the {\it subspace minimization} descent direction subproblem:
\begin{equation}\label{d}
	\min\limits_{d\in\Omega_{k}}\max\limits_{i\in[m]}q_{i}^{k}(d),
\end{equation}
where $q^{k}_{i}(\cdot)$ is approximation model for $F_{i}$ at $x^{k}$, and $\Omega_{k}$ is a subspace. The key issues for the descent direction are how to choose the subspaces and how to obtain the approximate models in corresponding subspaces quickly.
\subsection{Selection of subspace}
Nonlinear conjugate gradient methods utilize the current steepest descent direction and the previous descent direction to construct new descent direction. In order to compare with nonlinear conjugate gradient methods, we denote $\Omega_{k}=Span\{v_{k},d_{k-1}\}$ for $k>1$. The choice for $v_{k}$ is essential, recall that steepest descent direction often accepts a very small stepsize due to the imbalances between objectives, here we set $v_{k}$ the Barzilai-Borwein descent direction \cite{CTY2023a} at $x^{k}$, namely,
\begin{equation}\label{dbb}
	v_{k}:=\mathop{\arg\min}\limits_{v\in\mathbb{R}^{n}}\max\limits_{i\in[m]}\left\{\frac{\langle\nabla F_{i}(x^{k}),v\rangle}{\alpha_{i}^{k}}+\frac{1}{2}\nm{v}^{2}\right\},
\end{equation}
where $\alpha^{k}\in\mathbb{R}^{m}_{++}$ is given by Barzilai-Borwein method:
\begin{equation}\label{bbalpha_k}
	\alpha_{i}^{k}=\left\{
	\begin{aligned}
		&\max\left\{\alpha_{\min},\min\left\{\frac{\langle s_{k-1},y^{k-1}_{i}\rangle}{\nm{s_{k-1}}^{2}}, \alpha_{\max}\right\}\right\}, & \langle s_{k-1},y^{k-1}_{i}\rangle&>0, \\
		&\max\left\{\alpha_{\min},\min\left\{\frac{\nm{y^{k-1}_{i}}}{\nm{s_{k-1}}}, \alpha_{\max}\right\}\right\}, & \langle s_{k-1},y^{k-1}_{i}\rangle&<0, \\
		& \alpha_{\min}, &  \langle s_{k-1},y^{k-1}_{i}\rangle&=0,
	\end{aligned}
	\right.
\end{equation}
for all $i\in[m]$, where $\alpha_{\max}$ is a sufficient large positive constant and $\alpha_{\min}$ is a sufficient small positive constant, $s_{k-1}=x^{k}-x^{k-1},\ y^{k-1}_{i}=\nabla F_{i}({x^{k}})-\nabla F_{i}(x^{k-1}),\ i\in[m].$
\subsection{Selection of approximate model}
In general, iterative methods frequently leverage a quadratic model as it effectively approximates the objective function within a small neighborhood of the minimizer. Striving for a more optimal balance between computational cost and enhanced curvature exploration, we adopt the approximate model proposed in Chen et al. \cite{CTY2024p}:
\begin{equation}\label{bbvm}
	q^{k}_{i}(d):=\frac{\langle\nabla F_{i}(x^{k}),d\rangle}{\bar{\alpha}^{k}_{i}}+\frac{1}{2}\nm{d}^{2}_{B_{k}},
\end{equation}
where $\bar{\alpha}^{k}\in\mathbb{R}^{m}_{++}$ is as follows:
\begin{equation}\label{alpha_k}
	\bar{\alpha}^{k}_{i}=\left\{
	\begin{aligned}
		&\max\left\{\alpha_{\min},\min\left\{\frac{\langle s_{k-1},y^{k-1}_{i}\rangle}{\nm{s_{k-1}}^{2}_{B_{k}}}, \alpha_{\max}\right\}\right\}, & \langle s_{k-1},y^{k-1}_{i}\rangle&>0, \\
		&\max\left\{\alpha_{\min},\min\left\{\frac{\nm{y^{k-1}_{i}}}{\nm{B_{k}s_{k-1}}}, \alpha_{\max}\right\}\right\}, & \langle s_{k-1},y^{k-1}_{i}\rangle&<0, \\
		& \alpha_{\min}, &  \langle s_{k-1},y^{k-1}_{i}\rangle&=0,
	\end{aligned}
	\right.
\end{equation}
and $B_{k}$ is a positive definite matrix. From a preconditioning perspective, as described in \cite{CTY2024p}, a judicious choice for $B_{k}$ is to approximate the variable aggregated Hessian, i.e., $$B_{k}\approx\sum\limits_{i\in[m]}\frac{\bar{\lambda}^{k-1}_{i}}{\bar{\alpha}^{k-1}_{i}}\nabla^{2}F_{i}(x^{k}),$$ where $\bar{\lambda}^{k-1}\in\Delta_{m}$ the dual solution of (\ref{d}) at $x^{k-1}$. 
\subsection{Subspace minimization Barzilai-Borwein descent method}
Recall that $s_{k-1}=x^{k}-x^{k-1}=t_{k-1}d_{k-1}$, by substituting $d=\mu v_{k}+ \nu s_{k-1}$ into (\ref{d}), it can be reformulated as
\begin{equation}\label{d2}
	\min\limits_{(\mu,\nu)^{T}\in\mathbb{R}^{2}}\max\limits_{i\in[m]}\dual{\begin{pmatrix} \dual{\frac{\nabla F_{i}(x^{k})}{\bar{\alpha}^{k}_{i}},v_{k}}\\ \dual{\frac{\nabla F_{i}(x^{k})}{\bar{\alpha}^{k}_{i}},s_{k-1}}\end{pmatrix},\begin{pmatrix} \mu\\ \nu \end{pmatrix}}+\frac{1}{2}\dual{\begin{pmatrix} \mu\\ \nu \end{pmatrix},\begin{pmatrix} \dual{v_{k},B_{k}v_{k}} & \dual{v_{k},B_{k}s_{k-1}}\\ \dual{s_{k-1},B_{k}v_{k}} & \dual{s_{k-1},B_{k}s_{k-1}} \end{pmatrix}\begin{pmatrix} \mu\\ \nu \end{pmatrix}}.
\end{equation}
Recall that $B_{k}\approx\sum_{i\in[m]}\frac{\bar{\lambda}^{k-1}_{i}}{\bar{\alpha}^{k-1}_{i}}\nabla^{2}F_{i}(x^{k})$, to avoid calculating the matrix, we use finite differences of gradient to estimate matrix-vector products, i,e,,
$$B_{k}s_{k-1}\approx\sum\limits_{i\in[m]}\frac{\bar{\lambda}^{k-1}_{i}}{\bar{\alpha}^{k-1}_{i}}(\nabla F_{i}(x^{k})-\nabla F_{i}(x^{k-1})),$$
and 
$$B_{k}v_{k}\approx\sum\limits_{i\in[m]}\frac{\bar{\lambda}^{k-1}_{i}}{\bar{\alpha}^{k-1}_{i}}(\nabla F_{i}(x^{k})-\nabla F_{i}(x^{k}-v_{k})).$$
We denote $$y^{k-1}:=\sum\limits_{i\in[m]}\frac{\bar{\lambda}^{k-1}_{i}}{\bar{\alpha}^{k-1}_{i}}(\nabla F_{i}(x^{k})-\nabla F_{i}(x^{k-1})),$$ 
$$y^{k-1}_{v}:=\sum\limits_{i\in[m]}\frac{\bar{\lambda}^{k-1}_{i}}{\bar{\alpha}^{k-1}_{i}}(\nabla F_{i}(x^{k})-\nabla F_{i}(x^{k}-v_{k})),$$
and
\begin{equation}\label{H}
	H^{k}\approx\begin{pmatrix}\rho^{k}_{1}  &\dual{v_{k},y^{k-1}}\\ \dual{v_{k},y^{k-1}} & \rho^{k}_{2} \end{pmatrix},
\end{equation}
where $\rho^{k}_{1}\approx\dual{v_{k},y_{v}^{k-1}},~\rho^{k}_{2}\approx\dual{s_{k-1},y^{k-1}}.$ Then the {\it subspace minimization Barzilai-Borwein descent} direction $d_{k}=\mu_{k}v_{k}+\nu_{k}s_{k-1}$, where $(\mu_{k},\nu_{k})^{T}\in\mathbb{R}^{2}$ is the optimal solution of the following subproblem:
\begin{equation}\label{d3}
	\min\limits_{(\mu,\nu)^{T}\in\mathbb{R}^{2}}\max\limits_{i\in[m]}\dual{\begin{pmatrix} \dual{\frac{\nabla F_{i}(x^{k})}{\bar{\alpha}^{k}_{i}},v_{k}}\\ \dual{\frac{\nabla F_{i}(x^{k})}{\bar{\alpha}^{k}_{i}},s_{k-1}}\end{pmatrix},\begin{pmatrix} \mu\\ \nu \end{pmatrix}}+\frac{1}{2}\dual{\begin{pmatrix} \mu\\ \nu \end{pmatrix},H^{k}\begin{pmatrix} \mu\\ \nu \end{pmatrix}},
\end{equation}
where  $\bar{\alpha}^{k}\in\mathbb{R}^{m}_{++}$ is as follows:
\begin{equation}\label{alpha_k2}
	\bar{\alpha}^{k}_{i}=\left\{
	\begin{aligned}
		&\max\left\{\alpha_{\min},\min\left\{\frac{\langle s_{k-1},y^{k-1}_{i}\rangle}{\rho^{k}_{2}}, \alpha_{\max}\right\}\right\}, & \langle s_{k-1},y^{k-1}_{i}\rangle&>0, \\
		&\max\left\{\alpha_{\min},\min\left\{\frac{\nm{y^{k-1}_{i}}}{\nm{y^{k-1}}}, \alpha_{\max}\right\}\right\}, & \langle s_{k-1},y^{k-1}_{i}\rangle&<0, \\
		& \alpha_{\min}, &  \langle s_{k-1},y^{k-1}_{i}\rangle&=0.
	\end{aligned}
	\right.
\end{equation}
To ensure that $d_{k}$ is a descent direction, two conditions are required: $\rho^{k}_{2}>0$ and $H^{k}$ is positive definite. Here, we initially assume that these two conditions hold in (\ref{d3}).
Denote 
$$\theta(x^{k}):=	\min\limits_{(\mu,\nu)^{T}\in\mathbb{R}^{2}}\max\limits_{i\in[m]}\dual{\begin{pmatrix} \dual{\frac{\nabla F_{i}(x^{k})}{\bar{\alpha}^{k}_{i}},v_{k}}\\ \dual{\frac{\nabla F_{i}(x^{k})}{\bar{\alpha}^{k}_{i}},s_{k-1}}\end{pmatrix},\begin{pmatrix} \mu\\ \nu \end{pmatrix}}+\frac{1}{2}\dual{\begin{pmatrix} \mu\\ \nu \end{pmatrix},H^{k}\begin{pmatrix} \mu\\ \nu \end{pmatrix}},$$
and
$$\mathcal{D}_{\alpha}(x,d):=\max_{i\in[m]}\dual{\frac{\nabla F_{i}(x)}{\alpha_{i}},d}.$$
Indeed, problem (\ref{d3}) can be equivalently rewritten as the following smooth quadratic problem:
\begin{align*}\tag{QP}\label{QP}
	&\min\limits_{(t,d)\in\mathbb{R}\times\mathbb{R}^{n}}\ t+\frac{1}{2}\dual{\begin{pmatrix} \mu\\ \nu \end{pmatrix},H^{k}\begin{pmatrix} \mu\\ \nu \end{pmatrix}},\\
	&\ \ \ \ \ \ \mathrm{ s.t.} \ \ \ \ \dual{\begin{pmatrix} \dual{\frac{\nabla F_{i}(x^{k})}{\bar{\alpha}^{k}_{i}},v_{k}}\\ \dual{\frac{\nabla F_{i}(x^{k})}{\bar{\alpha}^{k}_{i}},s_{k-1}}\end{pmatrix},\begin{pmatrix} \mu\\ \nu \end{pmatrix}} \leq t,\ i\in[m].
\end{align*}
As described in \cite{CTY2023a}, the problem (\ref{QP}) can be efficiently solved via its dual. It is worth noting that $\bar{\lambda}^{k-1}$ should represent the dual solution of (\ref{QP}) at $x^{k-1}$ in this setting. By KKT conditions, we have
\begin{equation}\label{theta}
	\theta(x^{k}) = \frac{1}{2}\mathcal{D}_{\bar{\alpha}^{k}}(x^{k},d_{k}),
\end{equation}
and 
\begin{equation}\label{D}
	\mathcal{D}_{\bar{\alpha}^{k}}(x^{k},d_{k})=-\dual{\begin{pmatrix} \mu_{k}\\ \nu_{k} \end{pmatrix},H^{k}\begin{pmatrix} \mu_{k}\\ \nu_{k} \end{pmatrix}}.
\end{equation}
The remaining questions are: How do we ensure that $\rho^{k}_{2}>0$ and $H^{k}$ is positive definite?
\subsubsection{Selection of $\rho^{k}_{2}$}
 Note that in nonconvex cases $\dual{s_{k-1},y^{k-1}}\leq0$ can hold, then we set
\begin{equation}\label{rho2}
	\rho_{2}^{k}=\left\{
	\begin{aligned}
		\begin{split}
			&\dual{s_{k-1},y^{k-1}}, & &\dual{s_{k-1},y^{k-1}}>0, \\
			&\mathcal{D}_{\bar{\alpha}^{k-1}}(x^{k},s_{k-1})-\sum_{i\in[m]}\bar{\lambda}^{k-1}_{i}\dual{\nabla F_{i}(x^{k-1})/{\bar{\alpha}^{k-1}_{i}},s_{k-1}}, & &~~~~~~~~~\text{~otherwise}.
		\end{split}
	\end{aligned}
	\right.
\end{equation}
We introduce the following Wolfe line search to ensure $\rho^{k}_{2}>0$.
\begin{equation}\label{wolfe1}
	(F_{i}(x^{k}+td_{k})-F_{i}(x^{k}))/\bar{\alpha}^{k}_{i}\leq \sigma_{1} t\mathcal{D}_{\bar{\alpha}^{k}}(x^{k},d_{k}),~\forall i\in[m],
\end{equation}
\begin{equation}\label{wolfe2}
	\mathcal{D}_{\bar{\alpha}^{k}}(x^{k}+td_{k},d_{k})\geq\sigma_{2}\mathcal{D}_{\bar{\alpha}^{k}}(x^{k},d_{k}).
\end{equation}
To ensure the  Wolfe line search is well-defined, we require the following assumption.
\begin{assumption}\label{a2}
	For any $x^{0}\in\mathbb{R}^{n}$, the level set $\mathcal{L}_{F}(x^{0})=\{x:\ F(x)\preceq F(x^{0})\}$ is compact.
\end{assumption}
\begin{proposition}
	Suppose that Assumption \ref{a2} holds. Let $d_{k}$ is a descent direction, $0<\sigma_{1}\leq\sigma_{2}<1$. Then, there exists an interval $[t_{l},t_{u}]$, with $0<t_{l}<t_{u}$, such that for all $t\in[t_{l},t_{u}]$ equations (\ref{wolfe1}) and (\ref{wolfe2}) hold.
\end{proposition}
\begin{proof}
	The proof is similar to
	that in \cite[Proposition 2]{LM2023}, we omit it here.
\end{proof}
\begin{proposition}
	If the stepsize is obtained by Wolfe line search, then $\rho^{k}_{2}$ in (\ref{rho2}) is positive. 
\end{proposition}
\begin{proof}
	The assertions are obvious, we omit the proof here.
\end{proof}

\subsubsection{Selection of $H^{k}$}
To guarantee the positive definiteness of $H^{k}$, adopting the strategy proposed by Yuan and Stoer \cite{YS1995}:
$$H^{k}=\begin{pmatrix}\rho^{k}_{1}  &\dual{v_{k},y^{k-1}}\\ \dual{v_{k},y^{k-1}} & \rho^{k}_{2} \end{pmatrix},$$
where $$\rho^{k}_{1}=\frac{2\dual{v_{k},y^{k-1}}^{2}}{\rho^{k}_{2}}.$$ Another powerful strategy proposed by Dai and Kou \cite{DK2016}: $$\rho^{k}_{1}=\tau^{k}\frac{\nm{y^{k-1}}^{2}}{\rho^{k}_{2}}\nm{v_{k}}^{2}(\tau^{k}>1).$$ However, it is important to note that both methods lack global convergence in non-convex cases.  In global convergence analysis we will require the {\it sufficient descent condition} \cite{LP2018}:
\begin{equation}
	\mathcal{D}_{\bar{\alpha}^{k}}(x^{k},d_{k})\leq -c\nm{v_{k}}^{2},
\end{equation}
for some $c>0$ and for all $k\geq0$. Motivated by \cite{LLL2024}, we provide a sufficient condition on $H^{k}$ to ensure that the obtained search direction $d_{k}$ is a sufficient descent direction.

\begin{proposition}
	Let $\{H^{k}\}\in\mathbb{R}^{2\times2}$ be the sequence of symmetric matrices in (\ref{d3}) and assume that there exist constants $0<c_{1}\leq c_{2}$ such that
	\begin{equation}\label{cond}
		c_{1}\leq\lambda_{\min}(D_{k}^{-1}H^{k}D_{k}^{-1})\leq\lambda_{\max}(D_{k}^{-1}H^{k}D_{k}^{-1})\leq c_{2}
	\end{equation}
holds for all $k$, where
$$D_{k}=\begin{pmatrix} \nm{v_{k}} & 0\\ 0 & \nm{s_{k-1}} \end{pmatrix}.$$
Let $d_{k}=\mu_{k}v_{k}+\nu_{k}s_{k-1}$, where $(\mu_{k},\nu_{k})^{T}$ is the solution of (\ref{d3}). Then, the search direction $d_{k}$ satisfies the following conditions:
\begin{equation}\label{ine1}
	\mathcal{D}_{\bar{\alpha}^{k}}(x^{k},d_{k}) \leq -\frac{c_{1}}{2}\nm{d_{k}}^{2},
\end{equation} 

\begin{equation}\label{ine2}
\mathcal{D}_{\bar{\alpha}^{k}}(x^{k},d_{k}) \leq -\min\limits_{i\in[m]}(\frac{{\alpha}^{k}_{i}}{\bar{\alpha}^{k}_{i}})^{2}\frac{\nm{v_{k}}^{2}}{c_{2}},
\end{equation} 
and
\begin{equation}\label{ine3}
	\frac{\min_{i\in[m]}({{\alpha}^{k}_{i}}/{\bar{\alpha}^{k}_{i}})^{2}}{c_{2}\max_{i\in[m]}{{\alpha}^{k}_{i}}/{\bar{\alpha}^{k}_{i}}}\nm{v_{k}}\leq\nm{d_{k}}\leq\frac{2\max_{i\in[m]}{{\alpha}^{k}_{i}}/{\bar{\alpha}^{k}_{i}}}{c_{1}}\nm{v_{k}}.
\end{equation}
\end{proposition}
\begin{proof}
	From (\ref{D}), we derive that
	\begin{align*}
		-\mathcal{D}_{\bar{\alpha}^{k}}(x^{k},d_{k})&=\dual{\begin{pmatrix} \mu_{k}\\ \nu_{k} \end{pmatrix},H^{k}\begin{pmatrix} \mu_{k}\\ \nu_{k} \end{pmatrix}}\\
		&=\dual{D_{k}\begin{pmatrix} \mu_{k}\\ \nu_{k} \end{pmatrix},D_{k}^{-1}H^{k}D_{k}^{-1}D_{k}\begin{pmatrix} \mu_{k}\\ \nu_{k} \end{pmatrix}}\\
		&\geq c_{1}(\mu_{k}^{2}\nm{v_{k}}^{2}+\nu_{k}^{2}\nm{s_{k-1}}^{2})\\
		&\geq\frac{c_{1}}{2}\nm{d_{k}}^{2}.
	\end{align*}
This implies inequality (\ref{ine1}).
	We use the relation (\ref{cond}) to get
	\begin{align*}
		\theta(x^{k})&=\min\limits_{(\mu,\nu)^{T}\in\mathbb{R}^{2}}\max\limits_{i\in[m]}\dual{\begin{pmatrix} \dual{\frac{\nabla F_{i}(x^{k})}{\bar{\alpha}^{k}_{i}},v_{k}}\\ \dual{\frac{\nabla F_{i}(x^{k})}{\bar{\alpha}^{k}_{i}},s_{k-1}}\end{pmatrix},\begin{pmatrix} \mu\\ \nu \end{pmatrix}}+\frac{1}{2}\dual{\begin{pmatrix} \mu\\ \nu \end{pmatrix},H^{k}\begin{pmatrix} \mu\\ \nu \end{pmatrix}}\\
		&\leq\min\limits_{(\mu,\nu)^{T}\in\mathbb{R}^{2}}\max\limits_{i\in[m]}\dual{\begin{pmatrix} \dual{\frac{\nabla F_{i}(x^{k})}{\bar{\alpha}^{k}_{i}},v_{k}}\\ \dual{\frac{\nabla F_{i}(x^{k})}{\bar{\alpha}^{k}_{i}},s_{k-1}}\end{pmatrix},\begin{pmatrix} \mu\\ \nu \end{pmatrix}}+\frac{c_{2}}{2}\dual{D_{k}\begin{pmatrix} \mu\\ \nu \end{pmatrix},D_{k}\begin{pmatrix} \mu\\ \nu \end{pmatrix}}\\
		&\leq\min\limits_{\mu\in\mathbb{R}}\max\limits_{i\in[m]}\dual{\frac{\nabla F_{i}(x^{k})}{\bar{\alpha}^{k}_{i}},v_{k}}\mu+\frac{c_{2}}{2}\nm{v_{k}}^{2}\mu^{2}\\
		&\leq\min\limits_{\mu\in\mathbb{R}}\max\limits_{i\in[m]}-\frac{{\alpha}^{k}_{i}}{\bar{\alpha}^{k}_{i}}\nm{v_{k}}^{2}\mu+\frac{c_{2}}{2}\nm{v_{k}}^{2}\mu^{2}\\
		&=-\min\limits_{i\in[m]}(\frac{{\alpha}^{k}_{i}}{\bar{\alpha}^{k}_{i}})^{2}\frac{\nm{v_{k}}^{2}}{2c_{2}},
	\end{align*}
where the second inequality is given by setting $\nu=0$, and the second inequality is due to $\dual{\nabla F_{i}(x^{k}),v_{k}}\leq-\alpha^{k}_{i}\nm{v_{k}}^{2}$. Plugging the preceding bound into (\ref{theta}) gives inequality (\ref{ine2}). By the definition of $\mathcal{D}_{\alpha}(x,d)$, we have
\begin{equation}
	\begin{aligned}
		\mathcal{D}_{\bar{\alpha}^{k}}(x^{k},d_{k})&=\max_{i\in[m]}\frac{{\alpha}^{k}_{i}}{\bar{\alpha}^{k}_{i}}\dual{\frac{\nabla F_{i}(x^{k})}{\alpha^{k}_{i}},d_{k}}\\
		&\geq\max_{i\in[m]}\frac{{\alpha}^{k}_{i}}{\bar{\alpha}^{k}_{i}}\max_{i\in[m]}\dual{\frac{\nabla F_{i}(x^{k})}{\alpha^{k}_{i}},d_{k}}\\
		&\geq\max_{i\in[m]}\frac{{\alpha}^{k}_{i}}{\bar{\alpha}^{k}_{i}}\dual{\sum_{i\in[m]}{\lambda^{k}_{BB}}_{i}\frac{\nabla F_{i}(x^{k})}{\alpha^{k}_{i}},d_{k}}\\
		&\geq\max_{i\in[m]}\frac{{\alpha}^{k}_{i}}{\bar{\alpha}^{k}_{i}}\dual{-v_{k},d_{k}}\\
		&\geq-\max_{i\in[m]}\frac{{\alpha}^{k}_{i}}{\bar{\alpha}^{k}_{i}}\nm{v_{k}}\nm{d_{k}},
	\end{aligned}
\end{equation}
where the first inequality is due to the fact that $\max_{i\in[m]}\dual{{\nabla F_{i}(x^{k})}/{\alpha^{k}_{i}},d_{k}}\leq0$.
By substituting the latter bound into (\ref{ine1}) and (\ref{ine2}), respectively, we derive the relation (\ref{ine3}).
\end{proof}

\begin{remark}
	If $m=1$, by setting $\bar{\alpha}^{k}=\alpha^{k}=1$, the relations (\ref{ine2}) and (\ref{ine3}) reduce to (4.17) and (4.18) in \cite{LLL2024}, respectively.
\end{remark}

In addition to ensuring positive definiteness, the selected $H^{k}$ should also capture the problem's local curvature information in the low-dimensional subspace. Therefore, we set 
\begin{equation}\label{hk}
	H^{k}=\begin{pmatrix}\rho^{k}_{1}  &\dual{v_{k},y^{k-1}}\\ \dual{v_{k},y^{k-1}} & \rho^{k}_{2} \end{pmatrix},
\end{equation}
where
\begin{equation}\label{rho1}
	\rho_{1}^{k}=\left\{
	\begin{aligned}
		\begin{split}
			&\dual{v_{k},y^{k-1}_{v}}, & &\dual{v_{k},y^{k-1}_{v}}>0, \\
			&\nm{v_{k}}\nm{y_{v}^{k-1}}, & &~~~~~~\text{~otherwise}.
		\end{split}
	\end{aligned}
	\right.
\end{equation}
As described in \cite{LLL2024}, to guarantee $H^{k}$ satisfies  condition (\ref{cond}), we can proceed as follows:
\begin{algorithm}  
	\caption{{\ttfamily{modified\_Cholesky\_factorization}}}\label{cho}
	\LinesNumbered  \small
	\KwData{$D^{k},~0<c_{1}\leq c_{2}$}
	{Update $H^{k}$ as (\ref{hk})}\\
	{Set $\hat{H}^{k}=D_{k}^{-1}H^{k}D_{k}^{-1}$}\\
	{Compute a triangular matrix $L\in\mathbb{R}^{2\times2}$:
		\begin{equation*}
			L_{11}=\left\{
			\begin{aligned}
				\begin{split}
					&\sqrt{\hat{H}^{k}_{11}}, & &\sqrt{\hat{H}^{k}_{11}}>c_{1}, \\
					&\sqrt{c_{2}}, & &\text{~otherwise}.
				\end{split}
			\end{aligned}
			\right.
		\end{equation*}
	\begin{equation*}
		L_{21}=\frac{\hat{H}^{k}_{21}}{L_{11}},
	\end{equation*}
and 
\begin{equation*}
	L_{22}=\left\{
	\begin{aligned}
		\begin{split}
			&\sqrt{\hat{H}^{k}_{22}-L_{21}^{2}}, & &\hat{H}^{k}_{22}-L_{21}^{2}>c_{1}, \\
			&\sqrt{c_{2}}, & &\text{~otherwise}.
		\end{split}
	\end{aligned}
	\right.
\end{equation*}
	}\\
{Compute $\hat{H}^{k}=LL^{T}$}	\\
{Set $H^{k}=D_{k}\hat{H}^{k}D_{k}$}
 
\end{algorithm}

The subspace minimization Barzilai-Borwein descent method for MOPs is described as follows.
\begin{algorithm}  
	\caption{{\ttfamily{subspace\_minimization\_Barzilai-Borwein\_descent\_method\_for\_MOPs}}}\label{smbb}
	\LinesNumbered  
	\KwData{$x^{0}\in\mathbb{R}^{n},~0<c_{1}\leq c_{2},~0<\sigma_{1}\leq\sigma_{2}$}
	{Choose $x^{-1}$ in a small neighborhood of $x^{0}$}\\
	\For{$k=0,...$}{Update $\alpha^{k}_{i}$ as (\ref{bbalpha_k}),\ $i\in[m]$\\
		Compute $v_{k}$ and $\lambda^{k}$ as the solution and dual solution of (\ref{dbb}), respectively\\
		\eIf{$v_{k}=0$}{ {\bf{return}} Pareto critical point $x^{k}$ }{\eIf{$k=0$}{Set $d_{k}=v_{k}$,
				$\bar{\lambda}^{k}={\lambda}^{k}$,
				$\bar{\alpha}^{k}={\alpha}^{k}$}{Update $H^{k}$ by Algorithm \ref{cho}\\
				Update $\bar{\alpha}^{k}_{i}$ as (\ref{alpha_k2}),~$i\in[m]$\\
				Compute $(\mu_{k},\nu_{k})^{T}$ and $\bar{\lambda}^{k}$ as the solution and dual solution of (\ref{d3}), respectively\\
			Set $d_{k}=\mu_{k}v_{k}+\nu_{k}s_{k-1}$}
		Compute a stepsize $t_{k}$ satisfies equations (\ref{wolfe1}) and (\ref{wolfe2})\\
		$x^{k+1}:= x^{k}+t_{k}d_{k}$\\
	}}  
\end{algorithm}

\section{Convergence Analysis}\label{sec4}
This section presents the convergence results for Algorithm \ref{smbb}. Notably, Algorithm \ref{smbb} terminates with a Pareto critical point in a finite number of iterations or generates an infinite sequence of noncritical points. In the sequel, we will assume that Algorithm \ref{smbb} produces an infinite sequence of noncritical points. 
\subsection{Global Convergence}
In this subsection, we analyze the global convergence of Algorithm \ref{smbb} without making any convexity assumptions.
\begin{theorem}\label{t3}
		Suppose that Assumption \ref{a2} holds. Let $\{x^{k}\}$ be the sequence generated by Algorithm \ref{smbb}. Then $\{x_{k}\}$ has at least one accumulation point, and every accumulation point $x^{*}\in\mathcal{L}_{F}(x^{0})$ is a Pareto critical point.
\end{theorem}
\begin{proof}
	We use the relation (\ref{wolfe1}) to deduce that $\{F_{i}(x^{k})\}$ is monotone decreasing and that
	\begin{equation}\label{line}
		F_{i}(x^{k+1})-F_{i}(x^{k})\leq \alpha_{\min}\sigma_{1}t_{k}\mathcal{D}_{\bar{\alpha}^{k}}(x^{k},d_{k}).
	\end{equation}
	It follows that
	$\{x^{k}\}\subset\mathcal{L}_{F}(x^{0})$ and $\{x_{k}\}$ has at least one accumulation point $x^{*}$, namely, there exists an infinite index set $\mathcal{K}$ such that $\lim_{k\in\mathcal{K}}x^{k}=x^{*}$. From the compactness of $\mathcal{L}_{F}(x^{0})$ and continuity of $F$, we deduce that $\{F(x^{k})\}$ is bounded. This, together with the monotonicity of $\{F_{i}(x^{k})\}$, indicates that $\{F(x^{k})\}$ is a Cauchy sequence. Therefore, there exists a point $F^{*}$ such that
	$$\lim\limits_{k\rightarrow\infty}F(x^{k})=F^{*}=F(x^{*}).$$
	Summing the inequality (\ref{line}) from $k=0$ to infinity and substituting the preceding limit, we have 
	$$-\sum\limits_{k=0}^{\infty}\alpha_{\min}\sigma_{1}t_{k}\mathcal{D}_{\bar{\alpha}^{k}}(x^{k},d_{k})\leq F_{i}(x^{0})-F_{i}^{*}<\infty.$$
	Plugging relation (\ref{ine1}) into the latter inequality gives
	$$\sum\limits_{k=0}^{\infty}t_{k}\nm{d_{k}}^{2}<\infty.$$
	It follows that
	\begin{equation}\label{lim}
		\lim\limits_{k\in\mathcal{K}}t_{k}d_{k}=0.
	\end{equation}
We use relation (\ref{wolfe2}) to get
\begin{align*}
	(\sigma_{2}-1)\mathcal{D}_{\bar{\alpha}^{k}}(x^{k},d_{k})\leq \mathcal{D}_{\bar{\alpha}^{k}}(x^{k}+t_{k}d_{k},d_{k}) - \mathcal{D}_{\bar{\alpha}^{k}}(x^{k},d_{k}).
\end{align*}
Taking the limit on both sides, the latter inequality, together with (\ref{lim}) and the continuity of $\nabla F_{i}$, implies
$$\lim\limits_{k\in\mathcal{K}}\mathcal{D}_{\bar{\alpha}^{k}}(x^{k},d_{k})=0.$$
Plugging the above limit into (\ref{ine2}) gives
$$\lim\limits_{k\in\mathcal{K}}v_{k}=0.$$
It follows by the \cite[Lemma 5(d)]{CTY2023a} that $x^{*}$ is a Pareto critical point.
\end{proof}

\subsection{Linear convergence}
This subsection is devoted to the linear convergence of Algorithm \ref{smbb}. Before presenting the convergence result, we introduce the following error bound condition.
\begin{definition}\label{def4}
	The vector-valued function $F$ satisfies a global error bound, if there exists a constant $\kappa$ such that
	$$u_{0}(x)\leq\kappa\nm{v(x)}^{2},~\forall x\in\mathbb{R}^{n},$$
	where $$u_{0}(x):=\sup\limits_{y\in\mathbb{R}^{n}}\min\limits_{i\in[m]}\{F_{i}(x)-F_{i}(y)\}$$
	is a merit function for (\ref{MOP}) (see \cite[Theorem 3.1]{TFY2023n}).
\end{definition}
\begin{remark}
Since $\nm{v_{k}}$ and $\nm{d^{k}_{SD}}$ are equivalent, the definition \ref{def4} is equivalent to the multiobjective PL-inequality \cite{TFY2023n} for unconstrained multiobjective optimization problems. As a result, strong convexity of $F$ is a sufficient condition for the definition \ref{def4}.
\end{remark}

To establish the linear convergence result of SMBBMO, we must first derive a lower bound for the stepsize $t_{k}$.
\begin{assumption}\label{a3}
	For each $i\in[m]$, the gradient $\nabla F_{i}$ is Lipschitz continuous with constant $L_{i}$.
\end{assumption}
\begin{lemma}
	Suppose that Assumption \ref{a3} holds. If the stepsize $t_{k}$ is obtained by Wolfe line search, then
	\begin{equation}\label{step}
		t_{k}\geq t_{\min}:=\frac{(1-\sigma_{2})c_{1}\alpha_{\min}}{2L_{\max}},
	\end{equation}
	where $L_{\max}:=\max_{i\in[m]}\{L_{i}\}$.
\end{lemma}
\begin{proof}
	Using relation (\ref{wolfe2}) and Assumption \ref{a3}, we have
	\begin{align*}
		(\sigma_{2}-1)\mathcal{D}_{\bar{\alpha}^{k}}(x^{k},d_{k})&\leq \mathcal{D}_{\bar{\alpha}^{k}}(x^{k}+t_{k}d_{k},d_{k}) - \mathcal{D}_{\bar{\alpha}^{k}}(x^{k},d_{k})\\
		&\leq\max\limits_{i\in[m]}\dual{\frac{\nabla F_{i}(x^{k}+t_{k}d^{k})-\nabla F_{i}(x^{k})}{\bar{\alpha}^{k}_{i}},d^{k}}\\
		&\leq\max\limits_{i\in[m]}\frac{L_{i}}{\bar{\alpha}^{k}_{i}}t_{k}\nm{d^{k}}^{2}\\
		&\leq\frac{L_{\max}}{\alpha_{\min}}t_{k}\nm{d^{k}}^{2}.
	\end{align*}
	By substituting (\ref{ine1}) into the above inequality, the desired result follows.
\end{proof}

Next, we show the Q-linear convergence of $\{u_{0}(x^{k})\}$.
\begin{theorem}
	Suppose that $F$ satisfies definition \ref{def4} and Assumption \ref{a3} holds. Let $\{x^{k}\}$ be the sequence generated by Algorithm \ref{smbb}. Then
	$$u_{0}(x^{k+1})\leq\left(1-r\right)u_{0}(x^{k}),$$
	where $r:=\sigma_{1}t_{\min}\alpha_{\min}^{3}/(c_{2}\kappa\alpha_{\max}^{2})$.
\end{theorem}
\begin{proof}
	Using (\ref{line}) and (\ref{ine2}), we have
	\begin{align*}
			F_{i}(x^{k+1})-F_{i}(x^{k})&\leq \alpha_{\min}\sigma_{1}t_{k}\mathcal{D}_{\bar{\alpha}^{k}}(x^{k},d_{k})\\
			&\leq-\alpha_{\min}\sigma_{1}t_{\min}\min\limits_{i\in[m]}(\frac{{\alpha}^{k}_{i}}{\bar{\alpha}^{k}_{i}})^{2}\frac{\nm{v_{k}}^{2}}{c_{2}}\\
			&\leq-\sigma_{1}t_{\min}\alpha_{\min}^{3}/(c_{2}\alpha_{\max}^{2})\nm{v_{k}}^{2}\\
			&\leq-\sigma_{1}t_{\min}\alpha_{\min}^{3}/(c_{2}\kappa\alpha_{\max}^{2})u_{0}(x^{k}),
	\end{align*}
where the last inequality is due to the error bound. Denoting  $r:=\sigma_{1}t_{\min}\alpha_{\min}^{3}/(c_{2}\kappa\alpha_{\max}^{2})$, rearranging and taking the minimum and supremum with respect to $i\in[m]$ and $x\in\mathbb{R}^{n}$ on both
sides, respectively, we obtain
\begin{align*}
\max\limits_{x\in\mathbb{R}^{n}}\min\limits_{i\in[m]}\{F_{i}(x^{k+1})-F_{i}(x)\}\leq\max\limits_{x\in\mathbb{R}^{n}}\min\limits_{i\in[m]}\{F_{i}(x^{k})-F_{i}(x)\}-ru_{0}(x^{k}).
\end{align*}
Hence, the desired result follows.
\end{proof}
\section{Numerical results}\label{sec5}
In this section, we present numerical results to demonstrate the performance of SMBBMO for various problems. We also compare SMBBMO with Barzilai-Borwein descent method for MOPs (BBDMO) \cite{CTY2023a} and Barzilai-Borwein quasi-Newton method for MOPs (BBQNMO) \cite{CTY2024p} to show its efficiency. All numerical experiments were implemented in Python 3.7 and executed on a personal computer with an Intel Core i7-11390H, 3.40 GHz processor, and 16 GB of RAM. For BBDMO, BBQNMO and SMBBMO, we set $\alpha_{\min}=10^{-3}$ and $\alpha_{\max}=10^{3}$ to truncate the Barzilai-Borwein's parameter. We use the Wolfe line search as in algorithm 3 in \cite{LM2023}, and set $\sigma_{1}=10^{-4},~\sigma_{2}=0.1$ in Wolfe line search. To ensure that the algorithms terminate after a finite number of iterations, for all tested algorithms we use the stopping criterion:
 $$\theta(x)\geq-5 \times\texttt{eps}^{1/2},$$ 
 where $\theta(x)=-{1}/{2}\nm{v(x)}^{2}$ for BBDMO and SMBBMO, and $\theta(x)=-{1}/{2}\nm{d(x)}_{B(x)}^{2}$ for BBQNMO, respectively, and $\texttt{eps}=2^{-52}\approx2.22\times10^{-16}$  is the
 machine precision.
 We also set the maximum number of iterations to 500. For each problem, we use the same initial points for different tested algorithms. The initial points are randomly selected within the specified lower and upper bounds. Dual subproblems of different algorithms are efficiently solved by Frank-Wolfe method. The recorded averages from the 200 runs include the number of iterations, the number of function evaluations, and the CPU time.

\subsection{Ordinary test problems}
The tested algorithms are executed on several test problems, and the problem illustration is given in Table \ref{tab1}. The dimensions of variables and objective functions are presented in the second and third columns, respectively. $x_{L}$ and $x_{U}$ represent lower bounds and upper bounds of variables, respectively.
\begin{table}[H]
	\begin{center}
		\caption{Description of all test problems used in numerical experiments}
		\label{tab1}\vspace{-5mm}
	\end{center}
	\centering
	\resizebox{.8\columnwidth}{!}{
		\begin{tabular}{lllllllllll}
			\hline
			Problem    &  & $n$     &           & $m$     &  & $x_{L}$               &  & $x_{U}$             &  & Reference \\ \hline
			DD1        &  & 5     & \textbf{} & 2     &  & (-20,...,-20)   &  & (20,...,20)   &  & \cite{DD1998}         \\
			Deb        &  & 2     & \textbf{} & 2     &  & (0.1,0.1)       &  & (1,1)         &  & \cite{D1999}          \\
			Far1       &  & 2     & \textbf{} & 2     &  & (-1,-1)         &  & (1,1)         &  & \cite{BK1}         \\
			FDS        &  & 5     & \textbf{} & 3     &  & (-2,...,-2)     &  & (2,...,2)     &  & \cite{FD2009}         \\
			FF1        &  & 2     & \textbf{} & 2     &  & (-1,-1)         &  & (1,1)         &  & \cite{BK1}         \\
			Hil1       &  & 2     & \textbf{} & 2     &  & (0,0)           &  & (1,1)         &  & \cite{Hil1}         \\
			Imbalance1 &  & 2     & \textbf{} & 2     &  & (-2,-2)         &  & (2,2)         &  & \cite{CTY2023a}         \\
			Imbalance2 &  & 2     & \textbf{} & 2     &  & (-2,-2)         &  & (2,2)         &  & \cite{CTY2023a}         \\
			LE1        &  & 2     & \textbf{} & 2     &  & (-5,-5)         &  & (10,10)       &  & \cite{BK1}        \\
			PNR        &  & 2     & \textbf{} & 2     &  & (-2,-2)         &  & (2,2)         &  & \cite{PN2006}        \\
			VU1        &  & 2     & \textbf{} & 2     &  & (-3,-3)         &  & (3,3)         &  & \cite{BK1}        \\
			WIT1       &  & 2     & \textbf{} & 2     &  & (-2,-2)         &  & (2,2)         &  & \cite{W2012}       \\
			WIT2       &  & 2     & \textbf{} & 2     &  & (-2,-2)         &  & (2,2)         &  & \cite{W2012}        \\
			WIT3       &  & 2     & \textbf{} & 2     &  & (-2,-2)         &  & (2,2)         &  & \cite{W2012}        \\
			WIT4       &  & 2     & \textbf{} & 2     &  & (-2,-2)         &  & (2,2)         &  & \cite{W2012}        \\
			WIT5       &  & 2     & \textbf{} & 2     &  & (-2,-2)         &  & (2,2)         &  & \cite{W2012}        \\
			WIT6       &  & 2 & \textbf{} & 2 &  & (-2,-2)         &  & (2,2)         &  & \cite{W2012}        \\ \hline
		\end{tabular}
	}
\end{table}

\begin{table}[H]
	\begin{center}
		\caption{Number of average iterations (iter), number of average function evaluations (feval), and average CPU time (time($ms$)) of BBDMO, BBQNMO, and SMBBMO implemented on different test problems}\label{tab3}\vspace{-5mm}
	\end{center}
	\centering
	\resizebox{.8\columnwidth}{!}{
		\begin{tabular}{lrrrrrrrrrrrrrrr}
			\hline
			Problem &
			\multicolumn{3}{l}{BBDMO} &
			&
			\multicolumn{3}{l}{BBQNMO} &
			\multicolumn{1}{l}{} &
			\multicolumn{3}{l}{SMBBMO}  \\ \cline{2-4} \cline{6-8} \cline{10-12} 
			&
			iter &
			feval &
			time &
			\textbf{} &
			iter &
			feval &
			time &
			&
			iter &
			feval &
			time  \\ \hline
			DD1        & 5.77 & \textbf{5.91} & \textbf{1.36} &  & 7.82 & 16.09 & 2.87 &  & \textbf{5.38} & {8.08} & {1.60} \\
			Deb        & 3.53 & {5.59} & \textbf{0.96} &  & \textbf{3.17} & \textbf{4.51} & 1.40 &  & 3.28 & 5.81 & {1.04} \\
			Far1       & 32.07 & 32.56 & 7.18 &  & \textbf{6.94} & \textbf{16.11} & \textbf{2.74} &  & 15.24 & 35.96 & 7.89 \\
			FDS        & 4.12 & 4.35 & \textbf{2.60} &  & 4.54 & 5.77 & 4.90 &  & \textbf{3.83} & \textbf{4.23} & {4.87}  \\
			FF1        & 4.08 & {5.30} & \textbf{0.63} &  & \textbf{3.37} & \textbf{5.12} & 0.90 &  & 3.50 & 5.83 & {1.13} \\
			Hil1       & 9.19 & 9.96 & 1.46 &  & \textbf{3.85} & \textbf{7.26} & \textbf{1.13 }&  & 6.34 & 10.91 & {2.41}\\
			Imbalance1 & {2.55} & \textbf{3.48} & \textbf{0.40} &  & 2.46 & 7.33 & 0.62 &  & \textbf{2.00} & {4.86} & {0.62} \\
			Imbalance2 & \textbf{1.00} & \textbf{1.00} & 0.27 &  & \textbf{1.00} & \textbf{1.00} & 0.29 &  & \textbf{1.00} & \textbf{1.00} & \textbf{0.21} \\
			LE1        & {3.61} & \textbf{5.77} & \textbf{0.58} &  & 3.78 & 5.93 & 0.90 &  & \textbf{3.57} & 7.85 & {1.11}\\
			PNR        & {3.30} & \textbf{3.58} & \textbf{0.88} &  & 3.38 & 4.40 & 0.73 &  & \textbf{3.17} & {4.28} & {0.89}  \\
			VU1        & 13.68 & 13.73 & 1.86 &  & \textbf{7.73} & \textbf{12.41} & \textbf{1.70} &  & 11.49 & 16.47 & {3.32}\\
			WIT1       & {2.95} & 3.04 & \textbf{0.42} &  & 2.77 & 3.23 & 0.59 &  & \textbf{2.54} & \textbf{2.91} & {0.70} \\
			WIT2       & {3.27} & 3.37 & \textbf{0.48} &  & 3.09 & 3.23 & 0.68 &  & \textbf{2.81} & \textbf{2.99} & {0.76}  \\
			WIT3       & {4.17} & 4.26 & \textbf{0.59} &  & 3.87 & 3.97 & 0.80 &  & \textbf{3.52} & \textbf{3.77} & {1.02} \\
			WIT4       & {4.33} & {4.38} & \textbf{0.58} &  & 4.08 & 4.15 & 0.84 &  & \textbf{3.59} & \textbf{3.85} & {1.00}  \\
			WIT5       & {3.43} & {3.45} & \textbf{0.50} &  & 3.36 & 3.40 & 0.72 &  & \textbf{2.94} & \textbf{3.04} & {0.83}  \\
			WIT6       & \textbf{1.00} & \textbf{1.00} & \textbf{0.22} &  & \textbf{1.00} & \textbf{1.00} & {0.24} &  & \textbf{1.00} & \textbf{1.00} & {0.23} \\ \hline
		\end{tabular}
	}
\end{table}
\par For each test problem, Table \ref{tab3} presents the average number of iterations (iter), average function evaluations (feval), and average CPU time (time($ms$)) for the different algorithms. It is observed that BBQNMO and SMBBMO surpass BBDMO in terms of average iterations, suggesting their superior ability to capture the local geometry of the tested problems. Notably, SMBBMO demonstrates superior performance over BBQNMO, particularly when $n=2$; thus, SMBBMO effectively captures the local geometry of the problems across the entire space. However, compared to BBDMO and BBQNMO, SMBBMO shows a relatively poorer performance in CPU time. This can be attributed to the well-conditioning of the test problems and the necessity to solve two subproblems in SMBBMO.

\begin{figure}[H]
	\centering
	\subfigure[DD1]
	{
		\begin{minipage}[H]{.22\linewidth}
			\centering
			\includegraphics[scale=0.22]{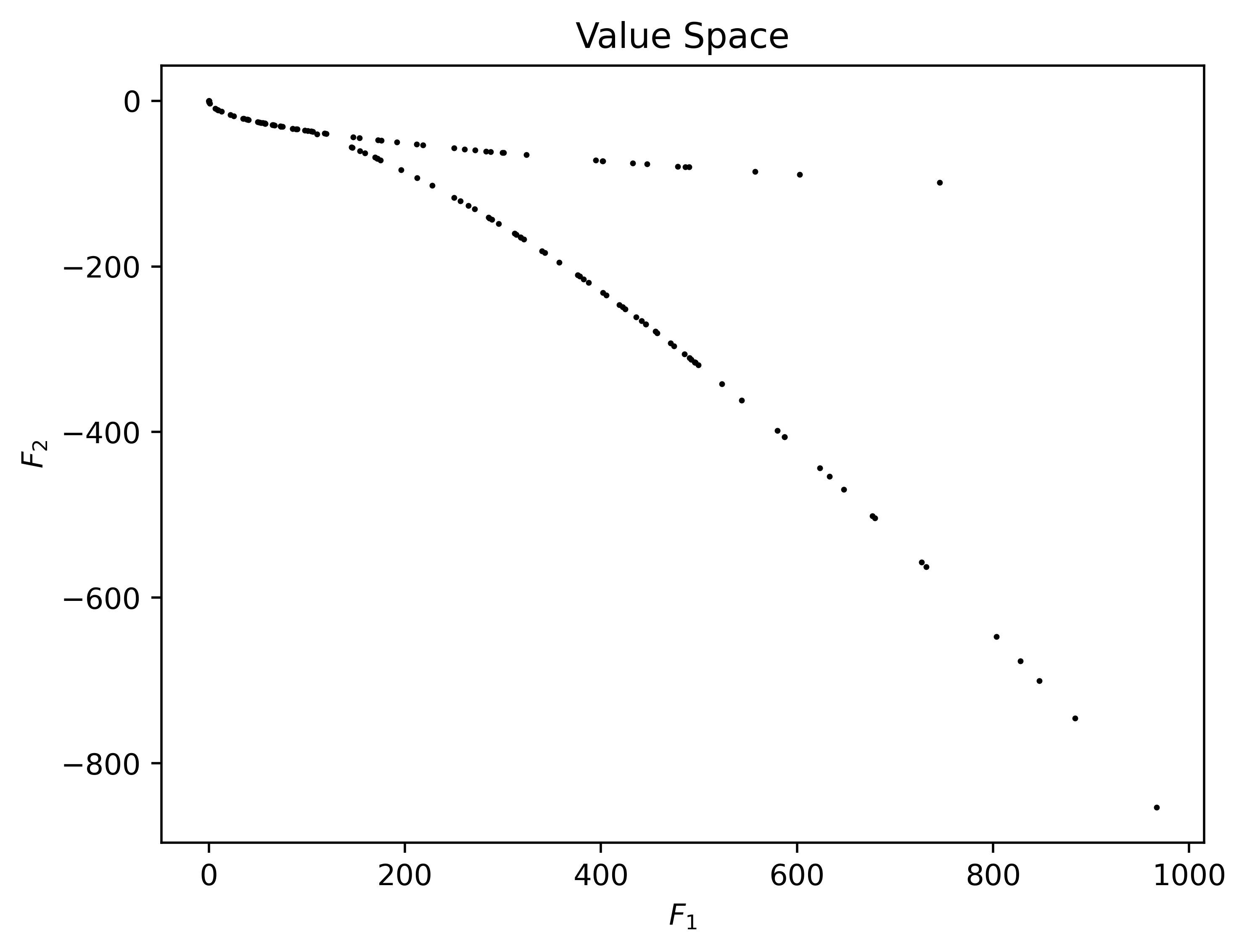} \\
			\includegraphics[scale=0.22]{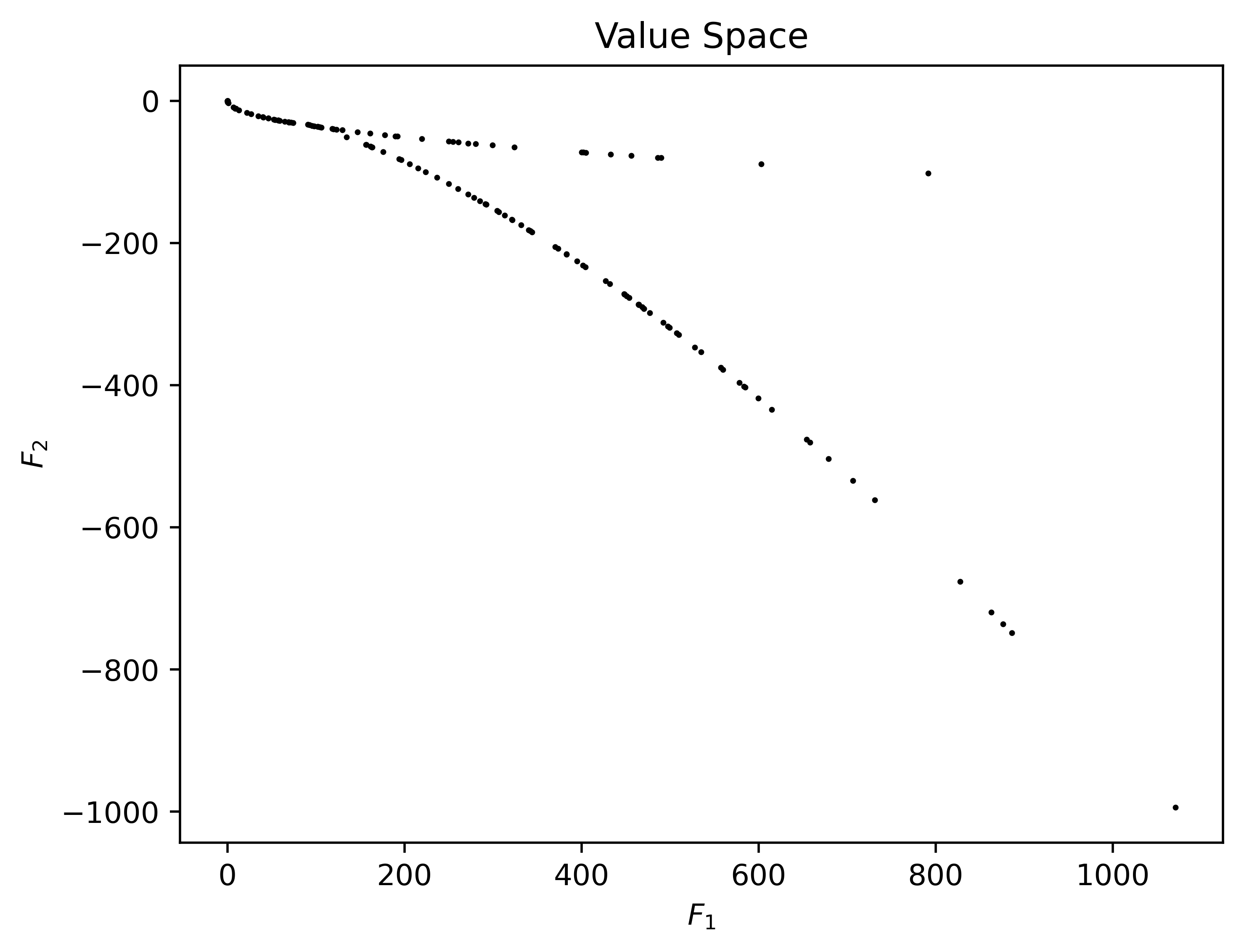} \\
			\includegraphics[scale=0.22]{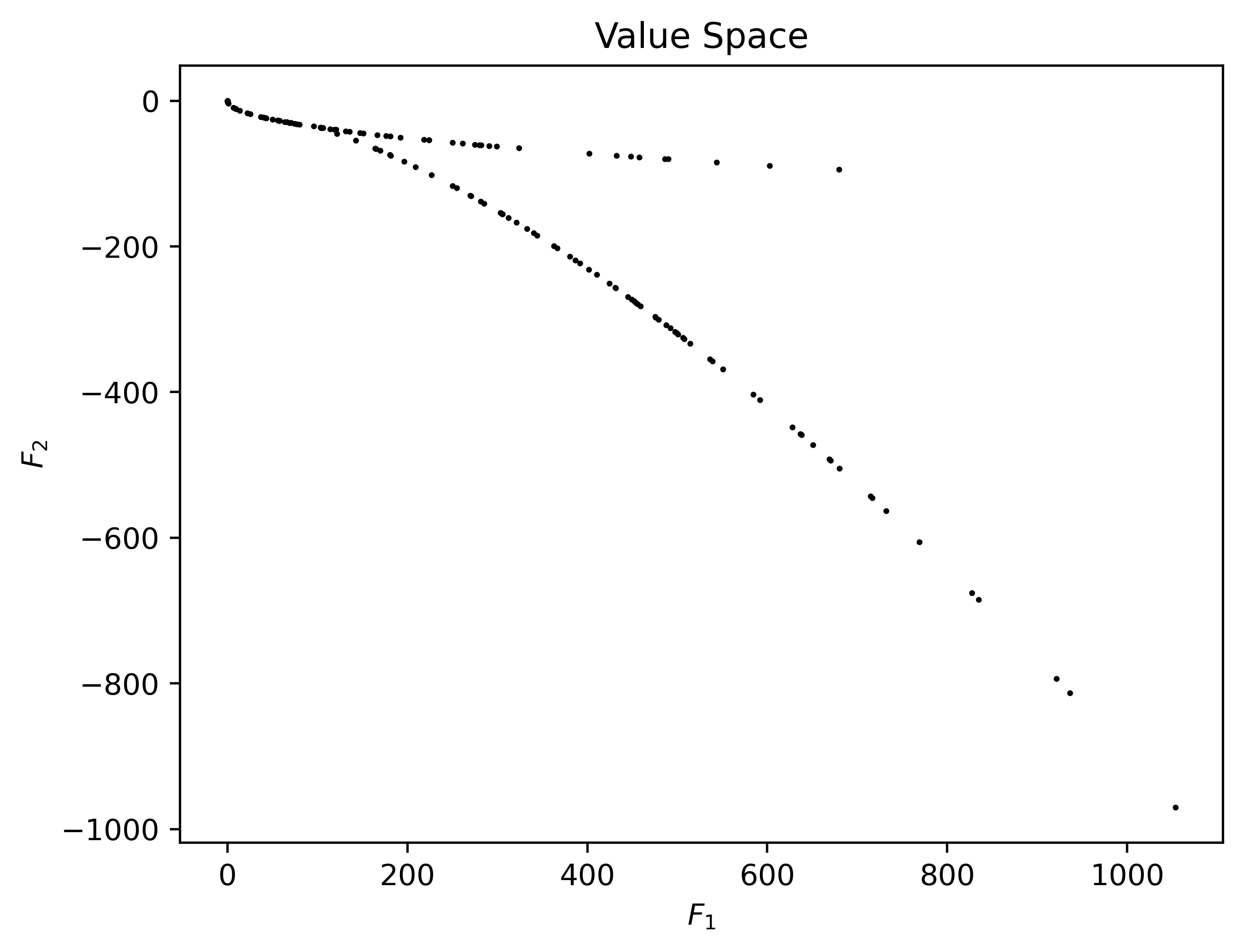}
		\end{minipage}
	}
	\subfigure[Hil1]
	{
		\begin{minipage}[H]{.22\linewidth}
			\centering
			\includegraphics[scale=0.22]{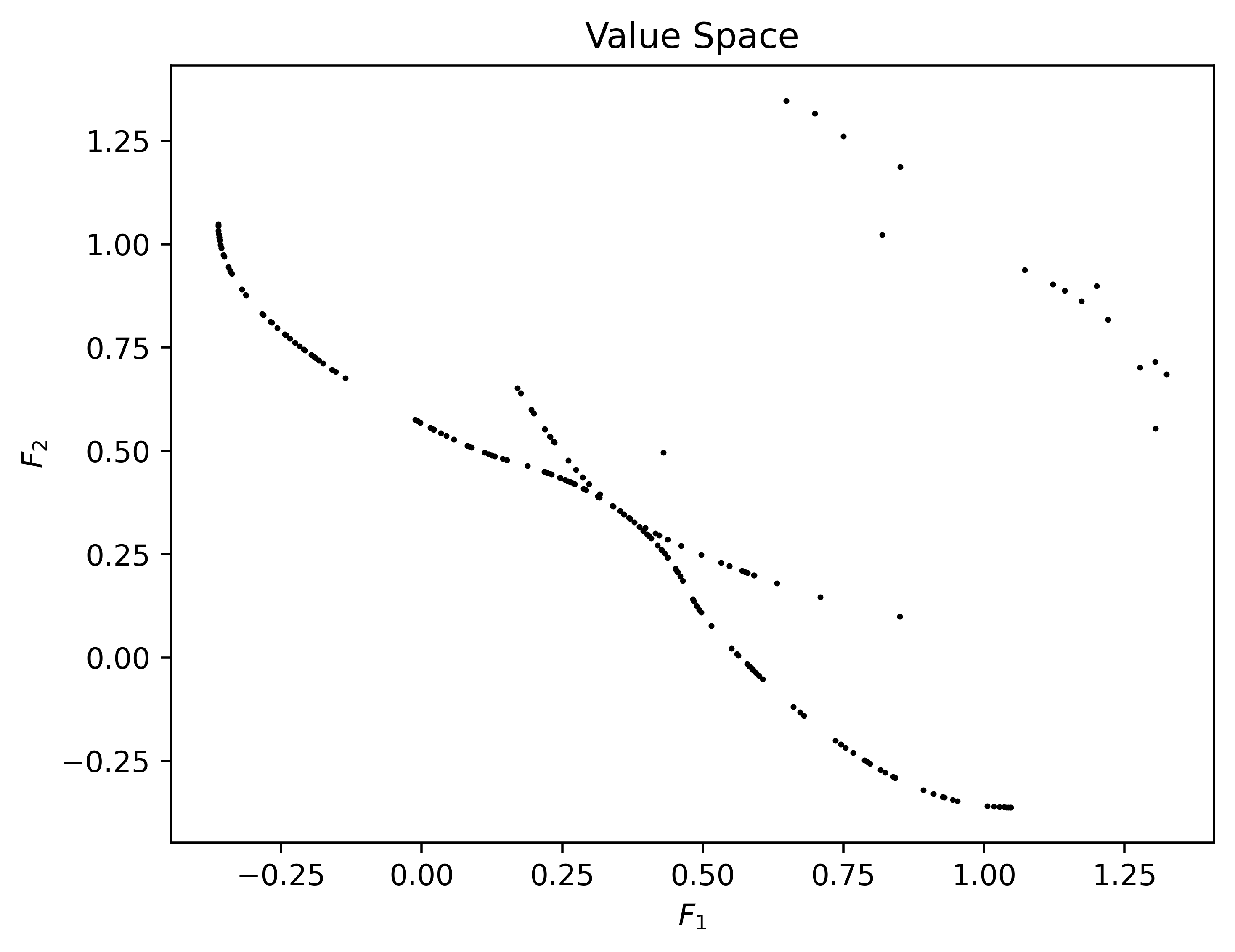} \\
			\includegraphics[scale=0.22]{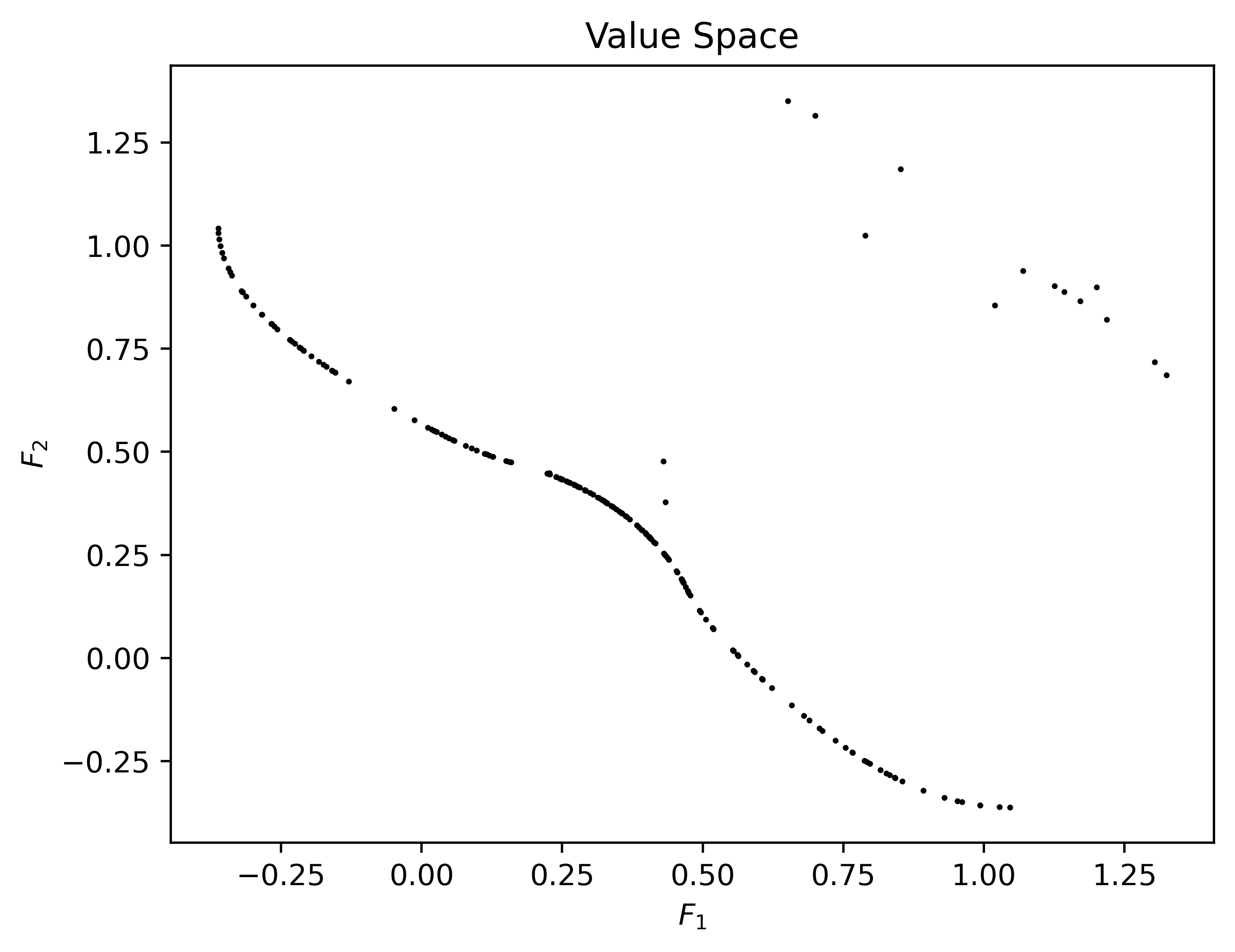} \\
			\includegraphics[scale=0.22]{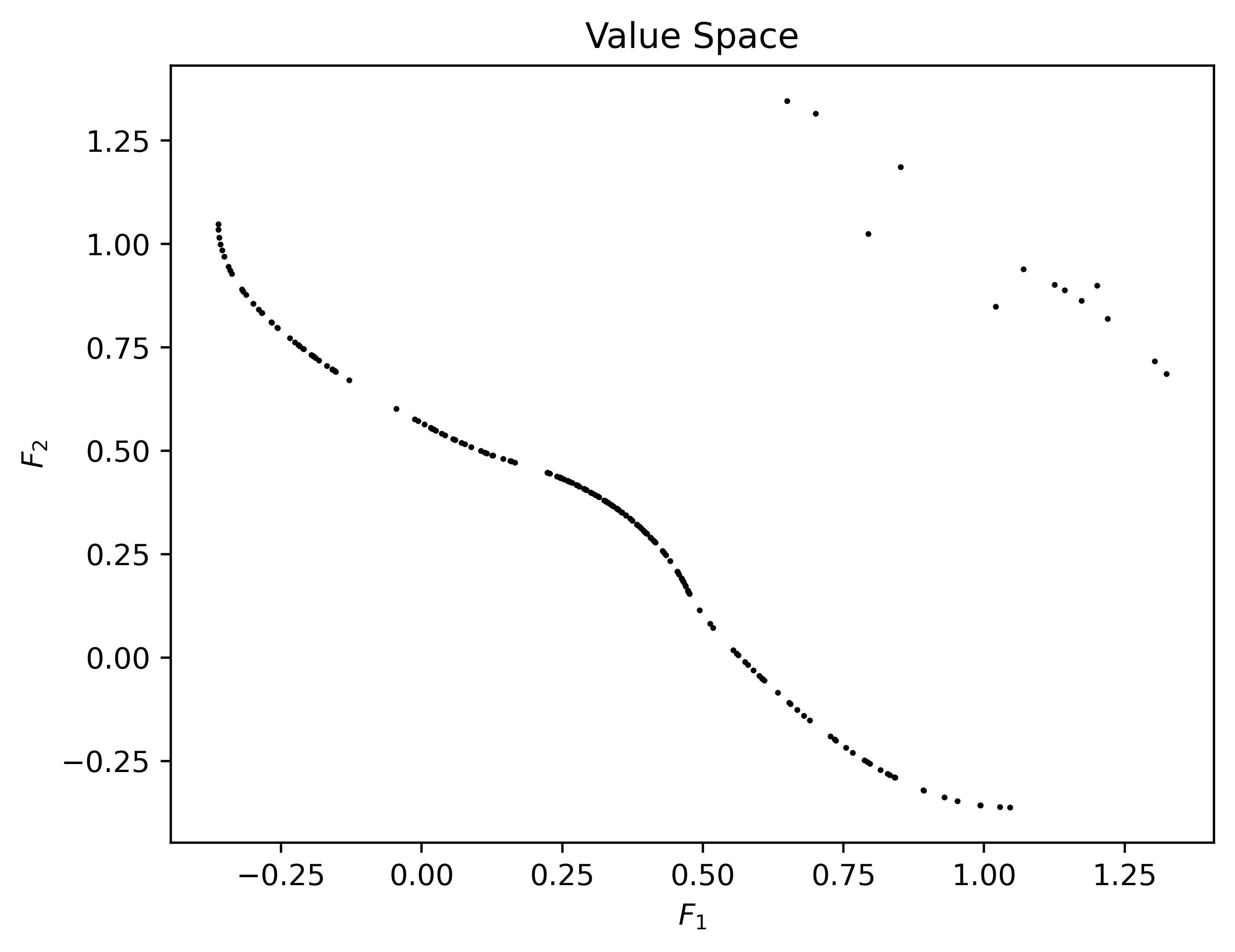}
		\end{minipage}
	}
    \subfigure[PNR]
    {
    	\begin{minipage}[H]{.22\linewidth}
    		\centering
    		\includegraphics[scale=0.22]{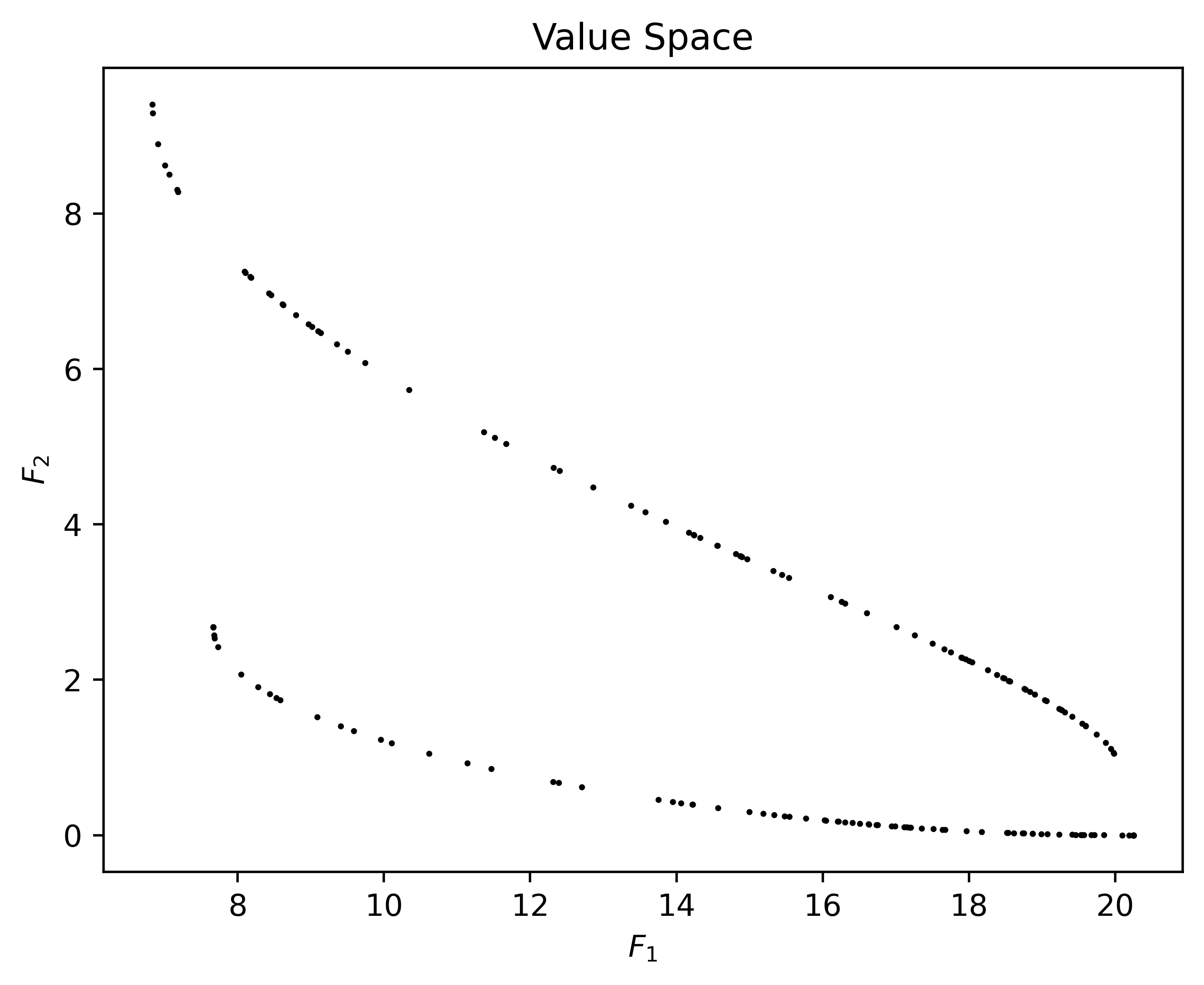} \\
    		\includegraphics[scale=0.22]{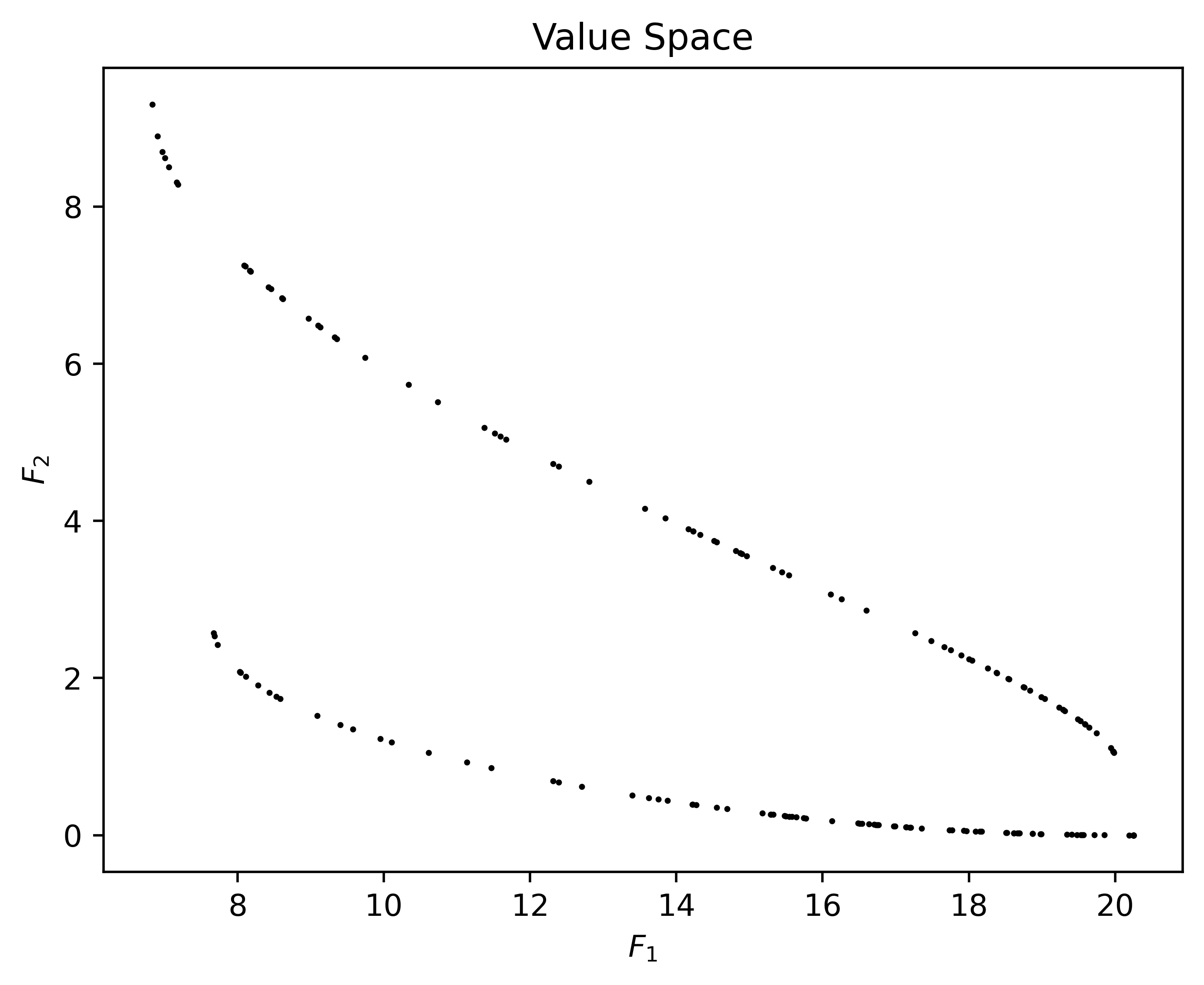} \\
    		\includegraphics[scale=0.22]{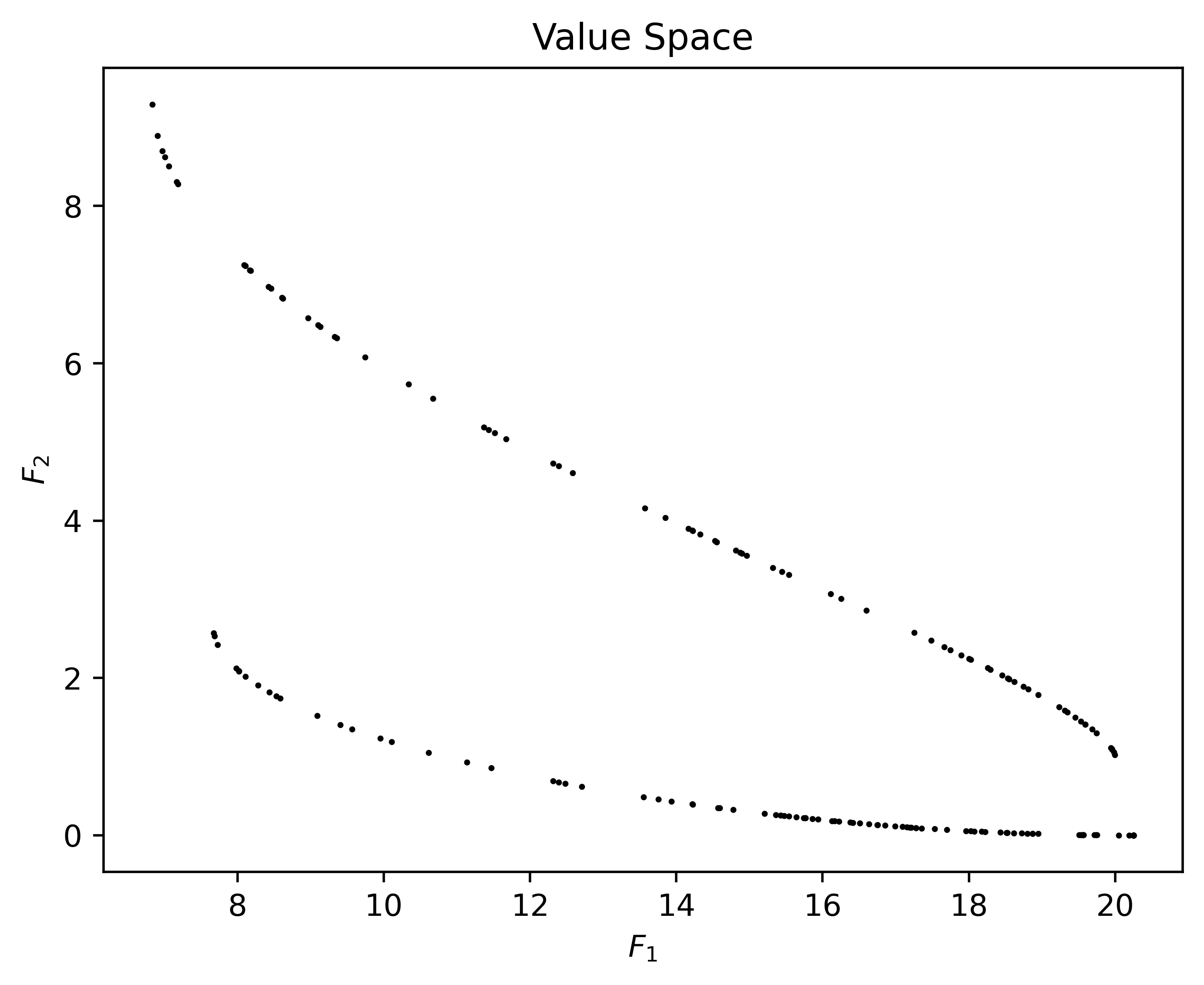}
    	\end{minipage}
    }
	\subfigure[VU1]
	{
		\begin{minipage}[H]{.22\linewidth}
			\centering
			\includegraphics[scale=0.22]{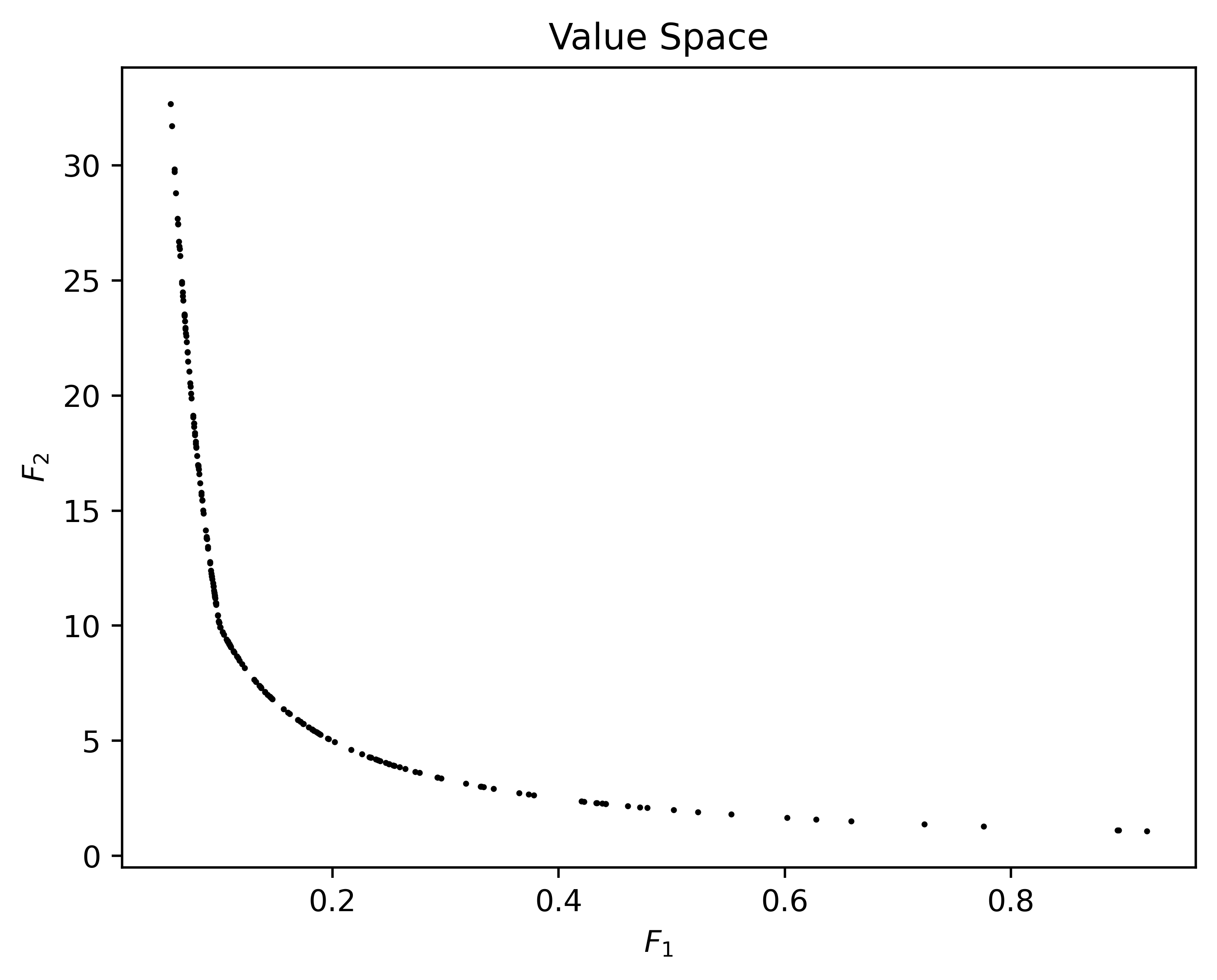} \\
			\includegraphics[scale=0.22]{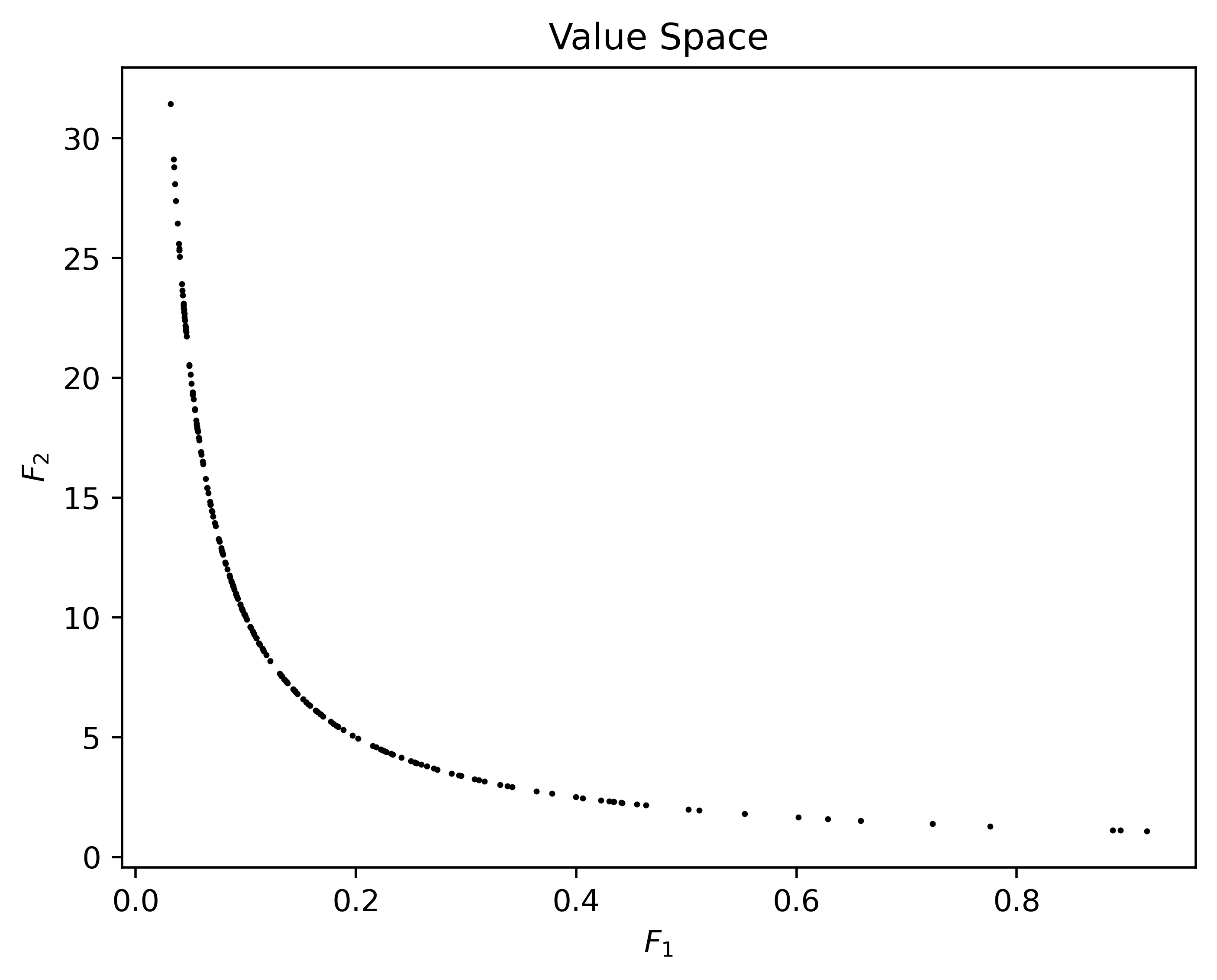} \\
			\includegraphics[scale=0.22]{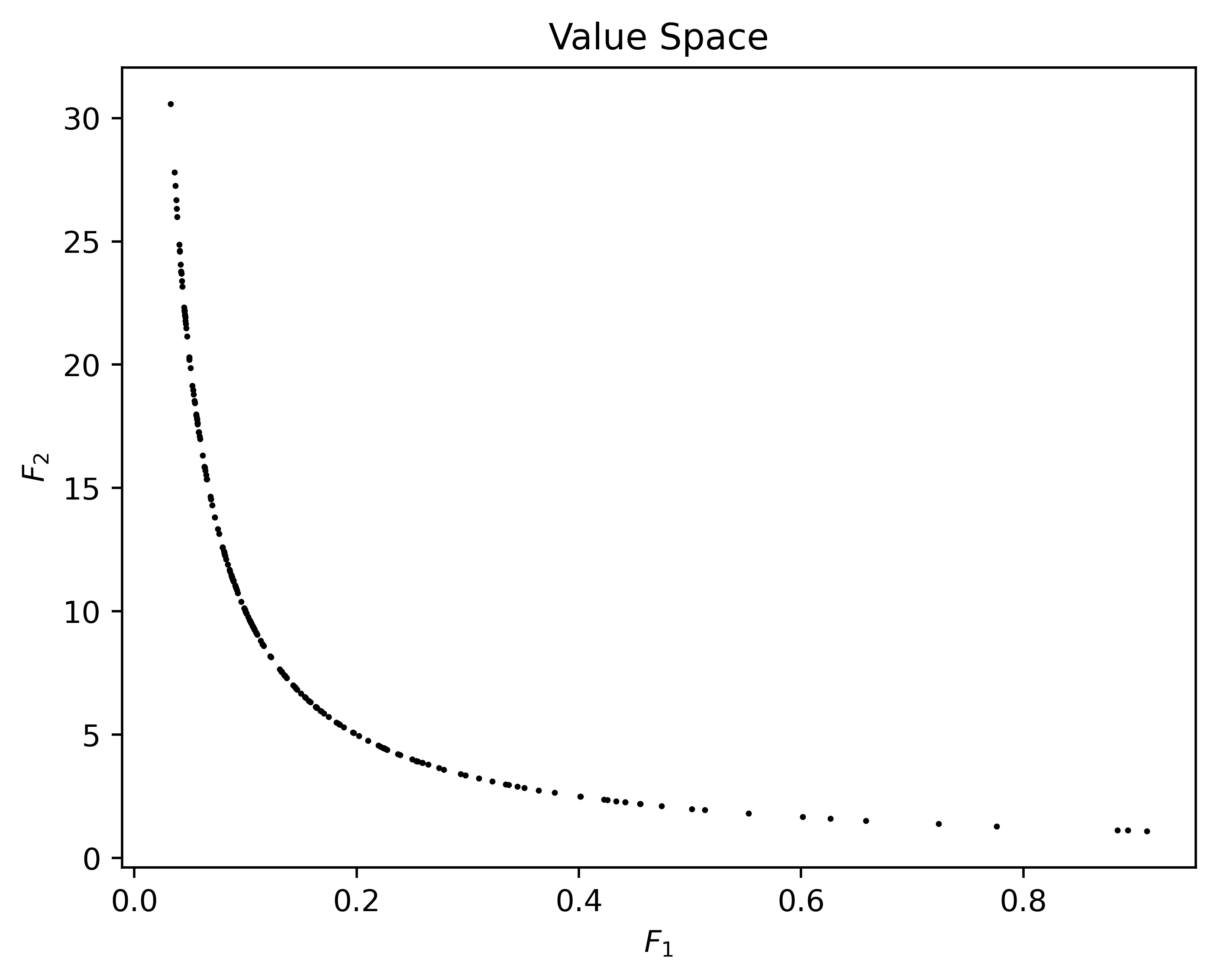}
		\end{minipage}
	}
	\caption{Numerical results in value space obtained by BBDMO ({\bf top}), BBQNMO ({\bf middle}) and SMBBMO for problems DD1, Hil1, PNR, and VU1.}
	\label{f2}
\end{figure}

%
 
\subsection{Quadratic ill-conditioned problems}
In this subsection, we evaluate the algorithm's performance on ill-conditioned problems. We consider a series of quadratic problems defined as follows:
$$F_{i}(x)=\frac{1}{2}\left\langle x,A_{i}x\right\rangle + \left\langle b_{i},x\right\rangle,~i=1,2,$$
where $A_{i}$ is a positive definite matrix. We set $A_{i}=H_{i}D_{i}H_{i}^{T}$, where $H_{i}$ is a random
orthogonal matrix and $D_{i}=Diag(d^{1}_{i},d^{2}_{i},...,d^{n}_{i})$ with $\max_{j}d^{j}_{i}/\min_{j}d^{j}_{i}=\kappa_{i}$. The problem illustration is given in Table \ref{tab2}. The second and third columns present the objective functions' dimension and condition numbers, respectively, while $x_L$ and $x_U$
represent the lower and upper bounds of the variables, respectively.

\begin{table}[H]
	\begin{center}
		\caption{Description of quadratic problems}
		\label{tab2}\vspace{-5mm}
	\end{center}
	\small
	\centering
	\begin{tabular}{clrlclrlrl}
		\hline
		\multicolumn{1}{l}{Problem} &  & n   &           & $(\kappa_{1},\kappa_{2})$ &           & \multicolumn{1}{c}{$x_{L}$} &  & \multicolumn{1}{c}{$x_{U}$} &  \\ \hline
		QPa                         &  & 10  &           & $(10,10)$ &           & 10{[}-1,...,-1{]}      &  & 10{[}1,...,1{]}        &  \\
		QPb                         &  & 10  &           & $(10^{2},10^{2})$ &           & 10{[}-1,...,-1{]}      &  & 10{[}1,...,1{]}        &  \\
		QPc                         &  & 100 &           & $(10^{2},10^{2})$ &           & 100{[}-1,...,-1{]}     &  & 100{[}1,...,1{]}       &  \\
		QPd                         &  & 100 &           & $(10^{3},10^{3})$ &           & 100{[}-1,...,-1{]}     &  & 100{[}1,...,1{]}       &  \\
		QPe                         &  & 500 &           & $(10^{3},10^{3})$ &           & 500{[}-1,...,-1{]}     &  & 500{[}1,...,1{]}       &  \\
		QPf                         &  & 500 & \textbf{} & $(10^{4},10^{4})$ & \textbf{} & 500{[}-1,...,-1{]}     &  & 500{[}1,...,1{]}       &  \\
		QPg                         &  & 1000 &           & $(10^{4},10^{4})$ &           & 1000{[}-1,...,-1{]}     &  & 1000{[}1,...,1{]}       &  \\
		QPh                         &  & 1000 &           & $(10^{5},10^{5})$ &           & 1000{[}-1,...,-1{]}     &  & 1000{[}1,...,1{]}       &  \\ \hline
	\end{tabular}
\end{table}

\begin{table}[h]
	\begin{center}
		\caption{Number of average iterations (iter), number of average function evaluations (feval), and average CPU time (time($ms$)) of BBDMO, BBQNMO,  and SMBBMO implemented on quadratic problems}\label{tab4}\vspace{-5mm}
	\end{center}
	\centering
	\resizebox{.95\columnwidth}{!}{
		\begin{tabular}{lrrrrrrrrrrrrrrr}
			\hline
			Problem &
			\multicolumn{3}{l}{BBDMO} &
			&
			\multicolumn{3}{l}{BBQNMO} &
			\multicolumn{1}{l}{} &
			\multicolumn{3}{l}{SMBBMO}  \\ \cline{2-4} \cline{6-8} \cline{10-12} 
			&
			\multicolumn{1}{r}{iter} &
			\multicolumn{1}{r}{feval} &
			\multicolumn{1}{r}{time} &
			\textbf{} &
			\multicolumn{1}{r}{iter} &
			\multicolumn{1}{r}{feval} &
			\multicolumn{1}{r}{time} &
			\multicolumn{1}{r}{} &
			\multicolumn{1}{r}{iter} &
			\multicolumn{1}{r}{feval} &
			\multicolumn{1}{r}{time}  \\ \hline
			QPa & 12.06 & 13.44 & \textbf{1.38} &  & \textbf{9.55} & 13.77 & {1.89} &  & 9.96 & \textbf{10.65} & 2.67  \\
			QPb & {42.24} & {67.46} & 5.04 &  & \textbf{20.16} & 38.92 & \textbf{4.32} &  & 20.67 & \textbf{31.60} & 5.92  \\
			QPc & 53.39 & 82.49 & \textbf{8.47} &  & \textbf{34.59} & 65.80 & 10.52 &  & {36.20} & \textbf{42.68} & {10.39}  \\
			QPd & 180.45 & 356.16 & 31.72 &  & \textbf{42.81} & 88.31 & \textbf{13.32} &  & 58.78 & \textbf{81.57} & 17.80  \\
			QPe & 184.43 & 343.49 & 111.07 &  & \textbf{64.94} & 110.92 & 830.70 &  & 81.60 & \textbf{87.49} & \textbf{47.68}  \\
			QPf & 436.72 & 1168.17 & 432.60 &  & \textbf{116.48} & 279.83 & 1286.07 &  & 121.55 & \textbf{203.98} & \textbf{80.80}  \\
			QPg & 320.00 & 909.17 & 1164.93 &  & 157.15 & 511.41 & 8483.27 &  & \textbf{154.84} & \textbf{189.42} & \textbf{468.84} \\ 
			QPh & 500.00 & 2856.25 & 3513.05 &  & \textbf{262.81} & 1106.15 & 15049.72 &  & 375.41 & \textbf{785.56} & \textbf{1542.79} \\ \hline
		\end{tabular}
	}
\end{table}

\begin{figure}[H]
	\centering
	\subfigure[QPe]
	{
		\begin{minipage}[H]{.22\linewidth}
			\centering
			\includegraphics[scale=0.22]{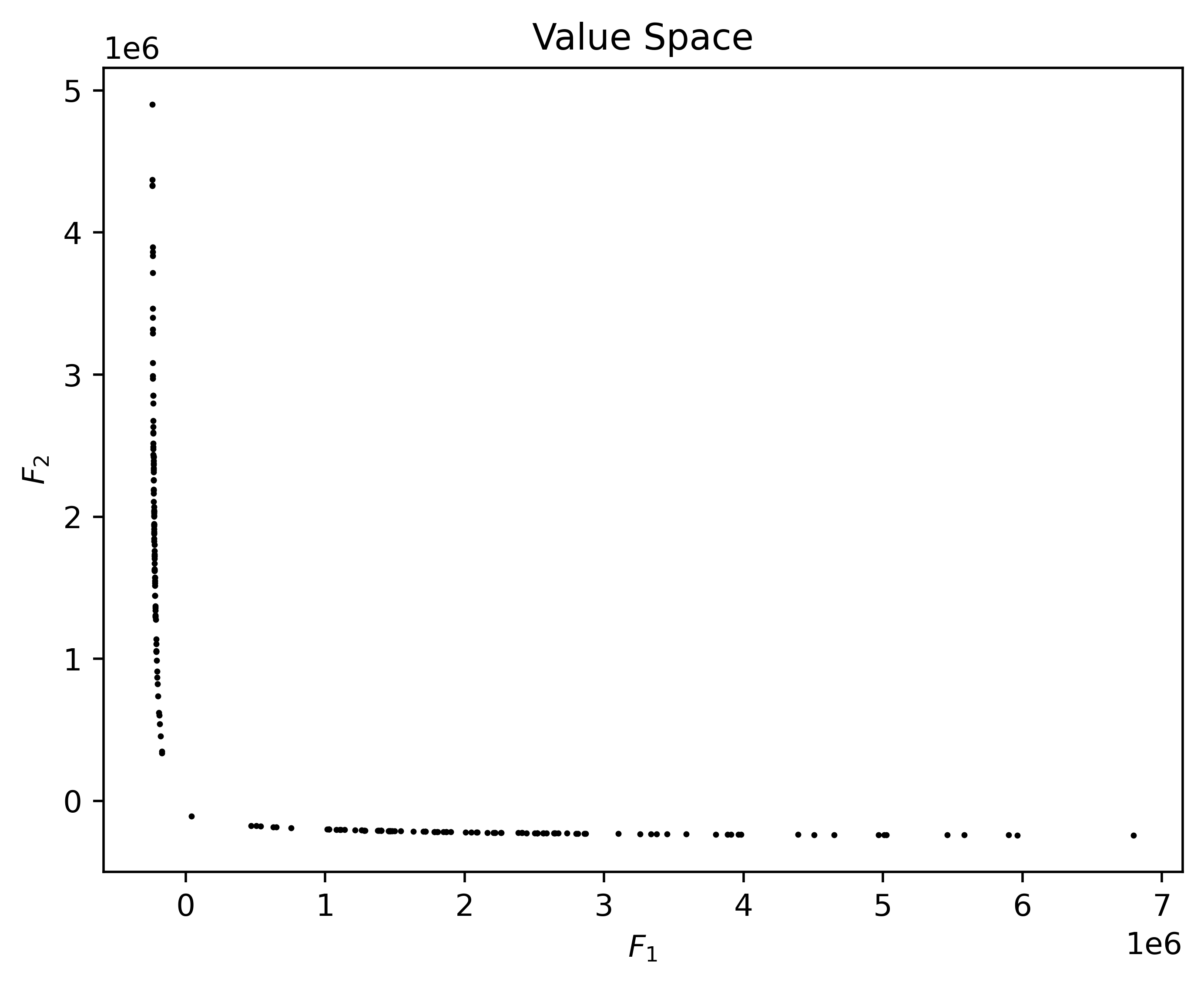} \\
			\includegraphics[scale=0.22]{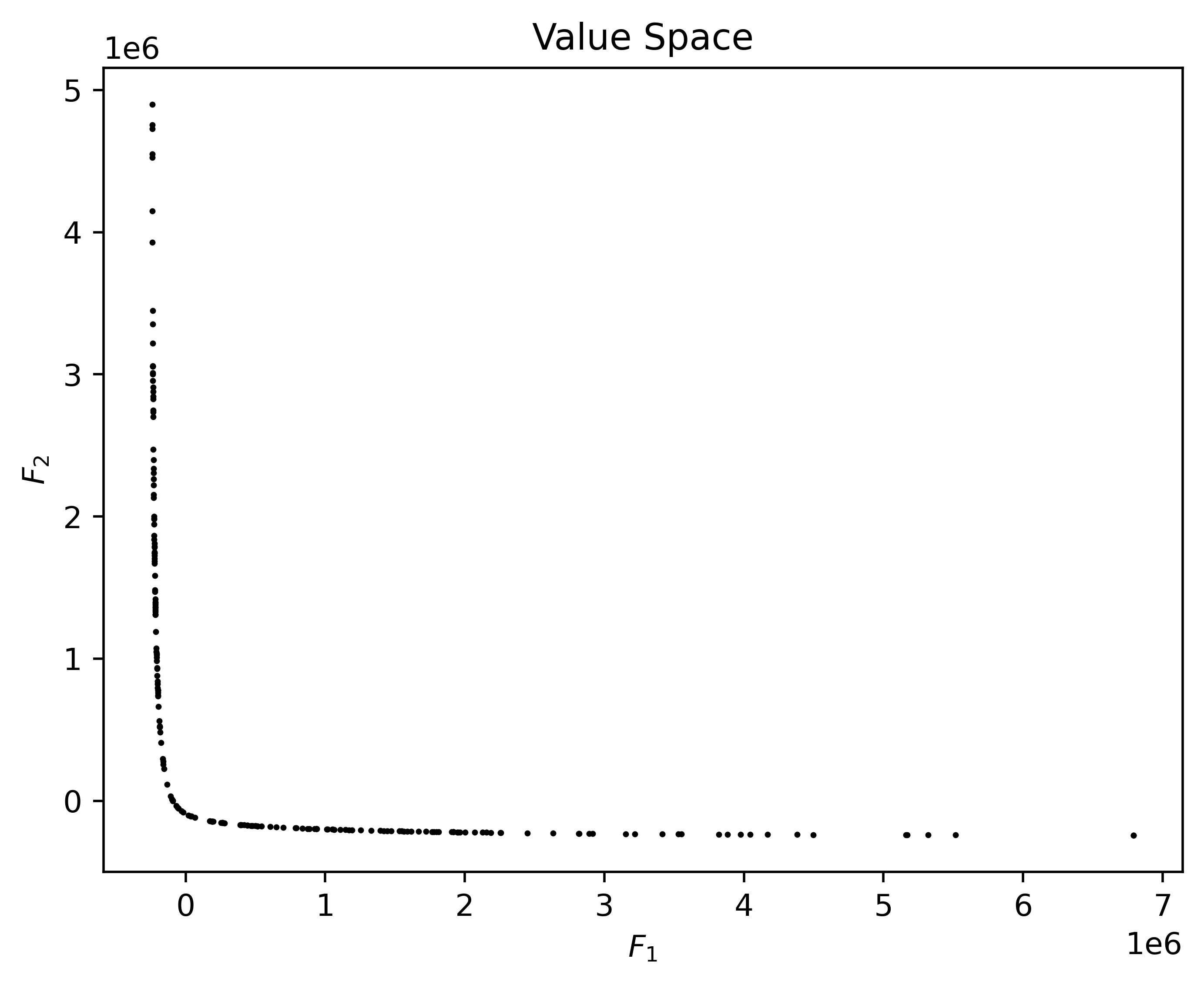} \\
			\includegraphics[scale=0.22]{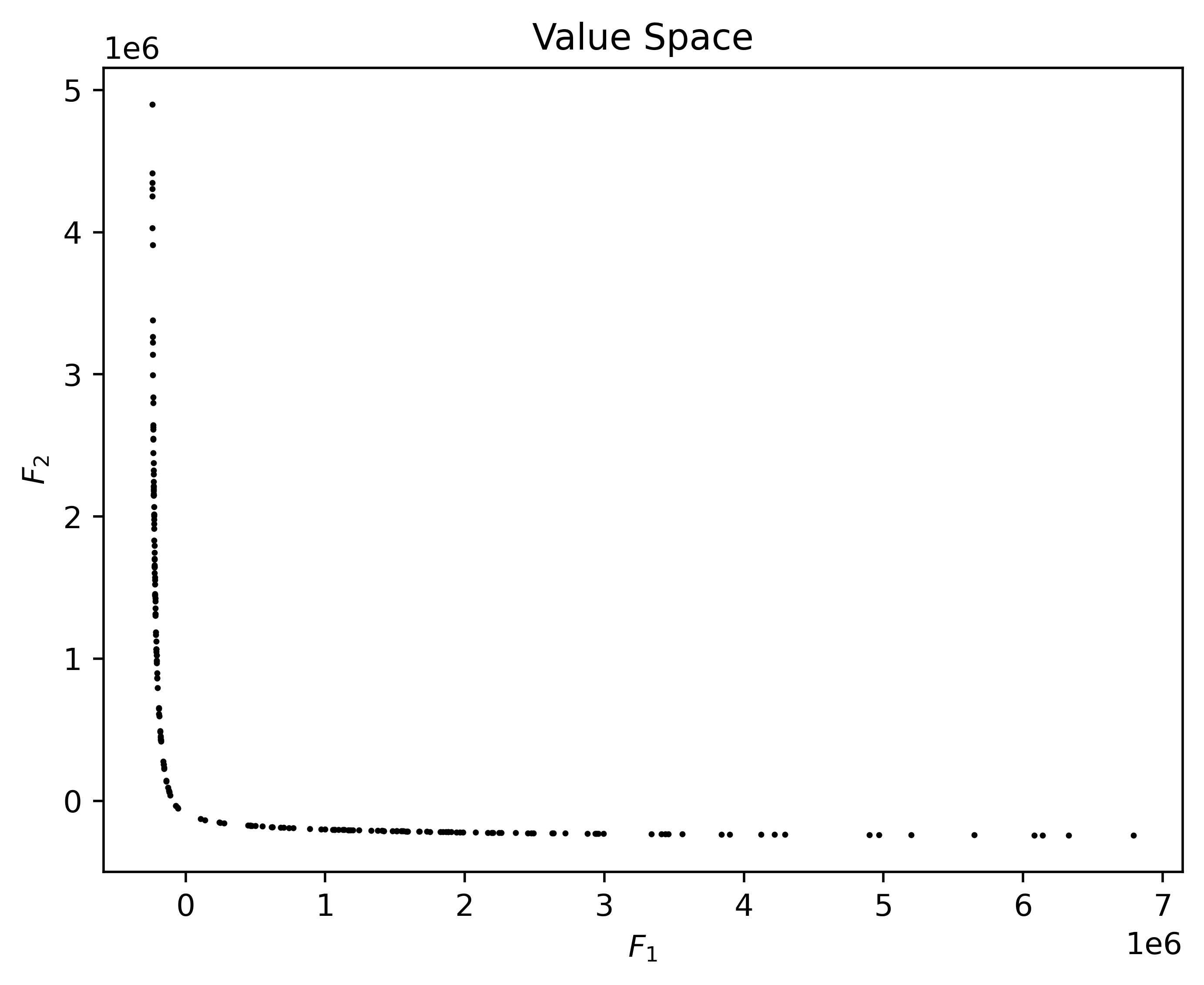}
		\end{minipage}
	}
	\subfigure[QPf]
	{
		\begin{minipage}[H]{.22\linewidth}
			\centering
			\includegraphics[scale=0.22]{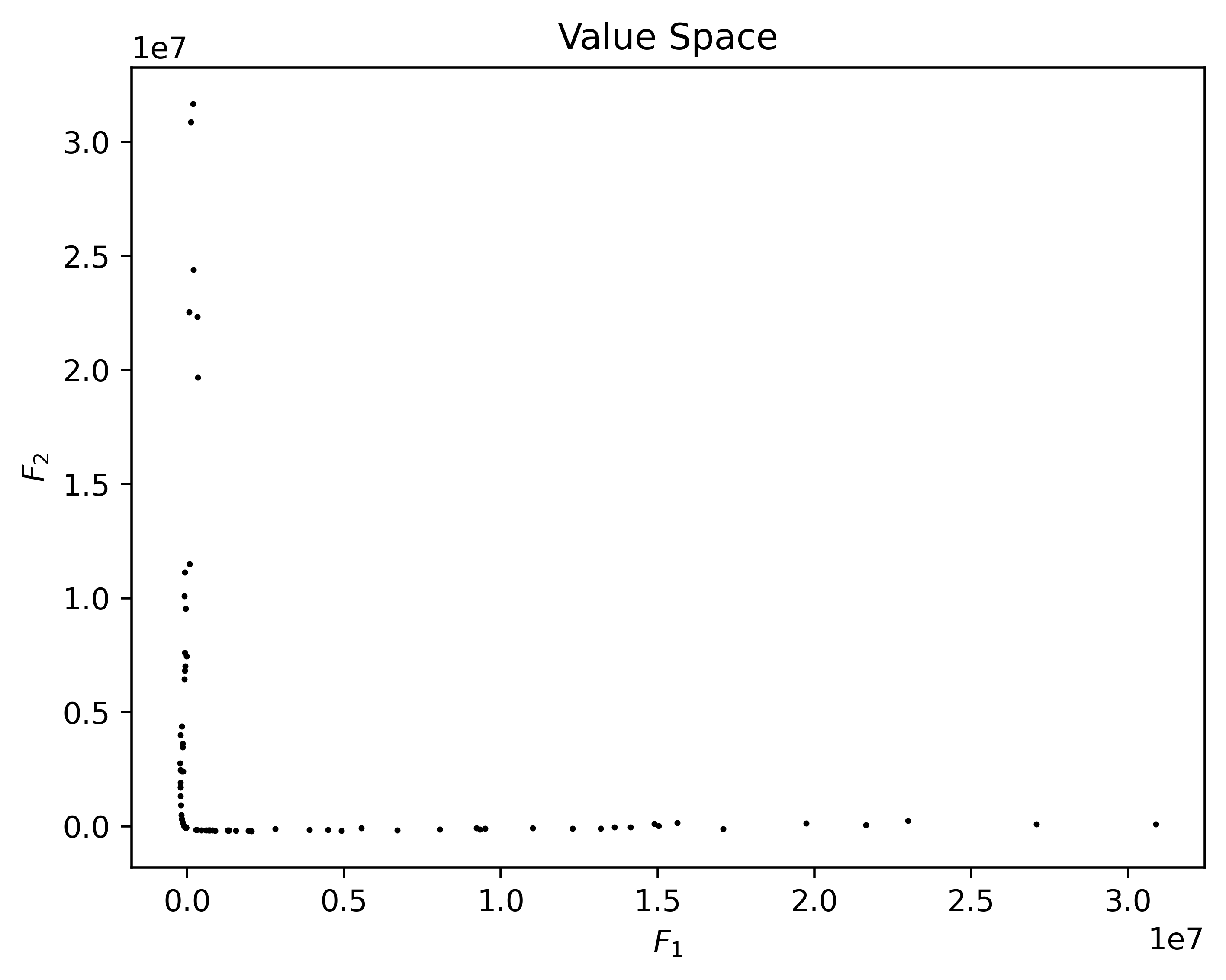} \\
			\includegraphics[scale=0.22]{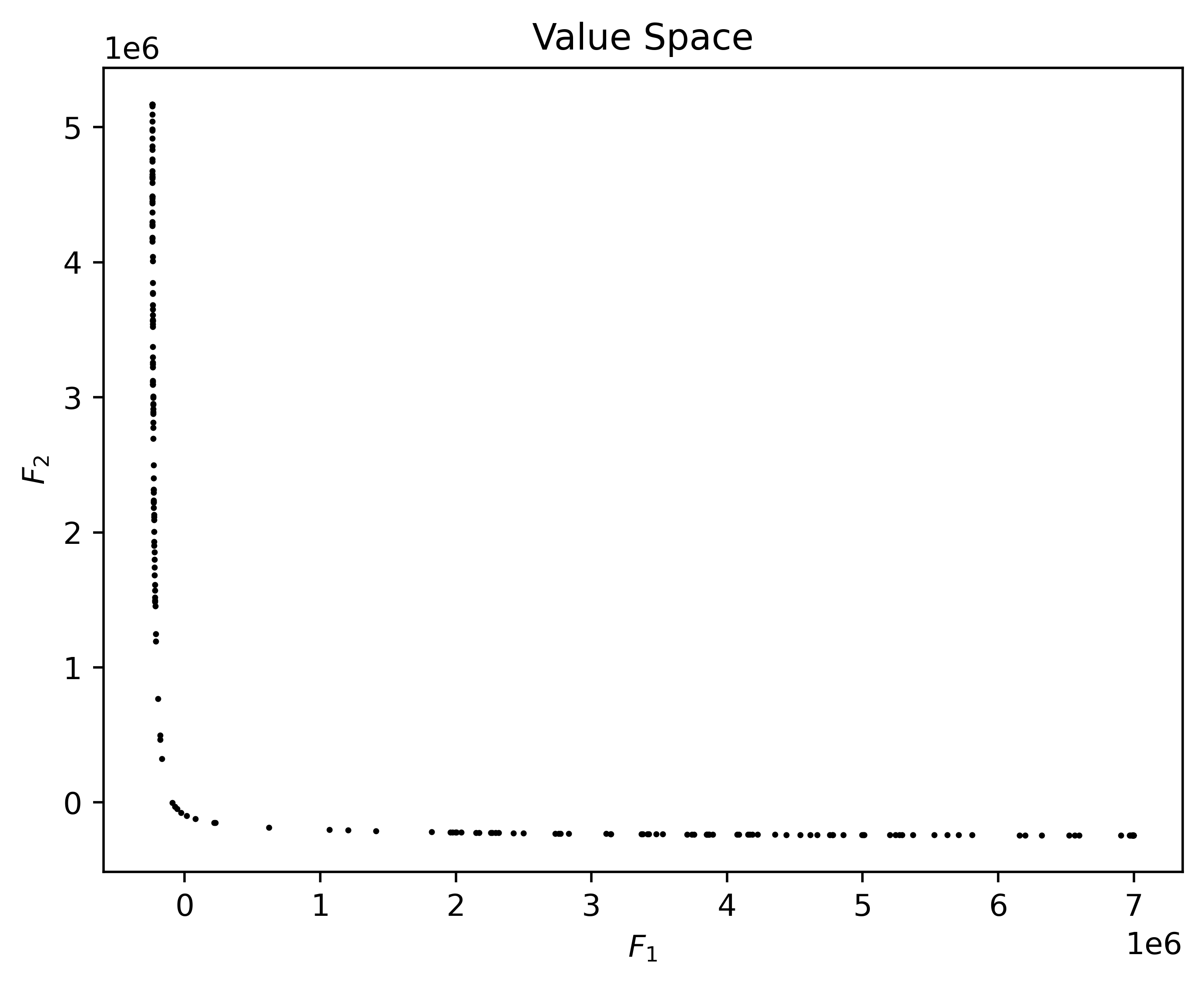} \\
			\includegraphics[scale=0.22]{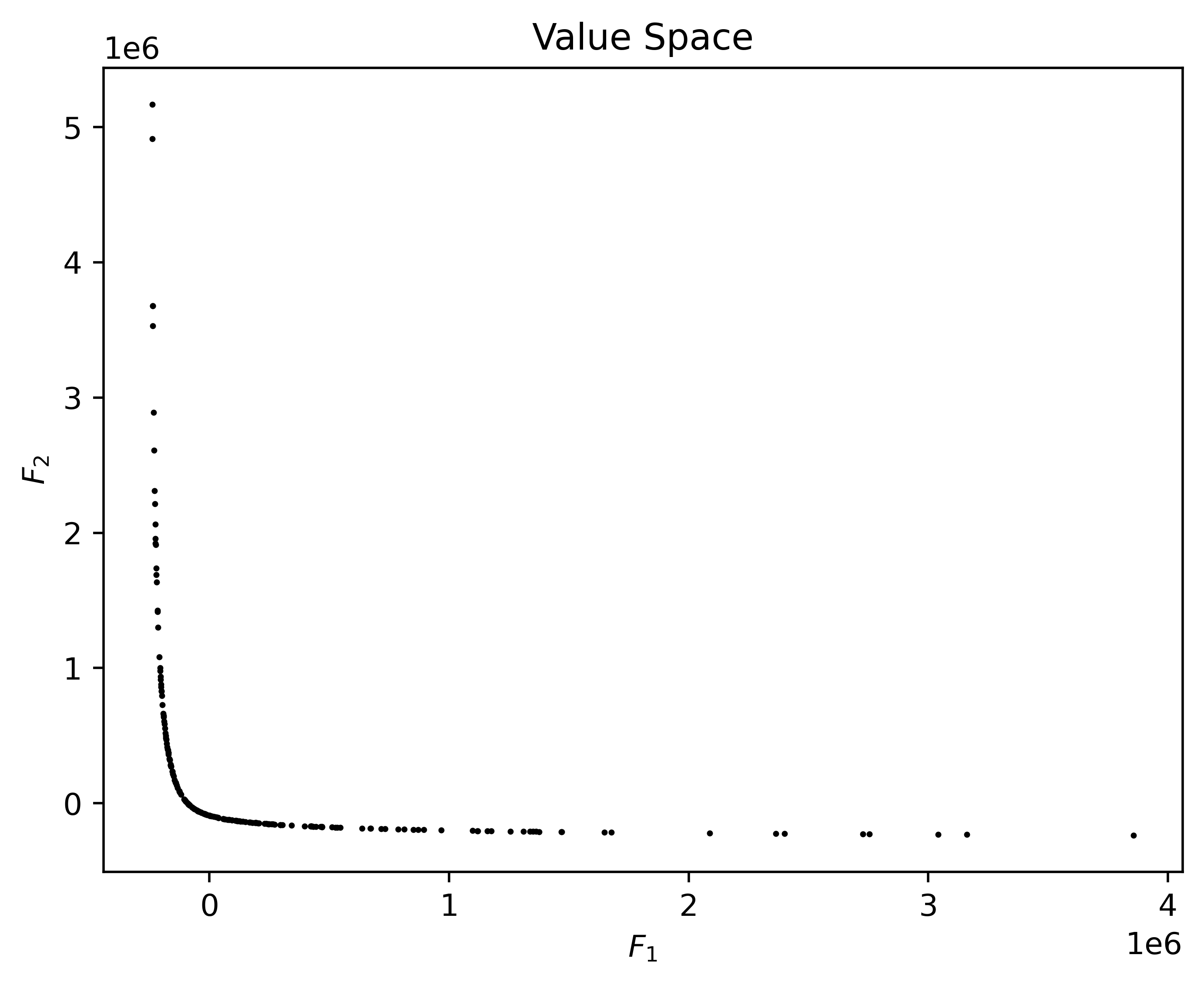}
		\end{minipage}
	}
	\subfigure[QPg]
	{
		\begin{minipage}[H]{.22\linewidth}
			\centering
			\includegraphics[scale=0.22]{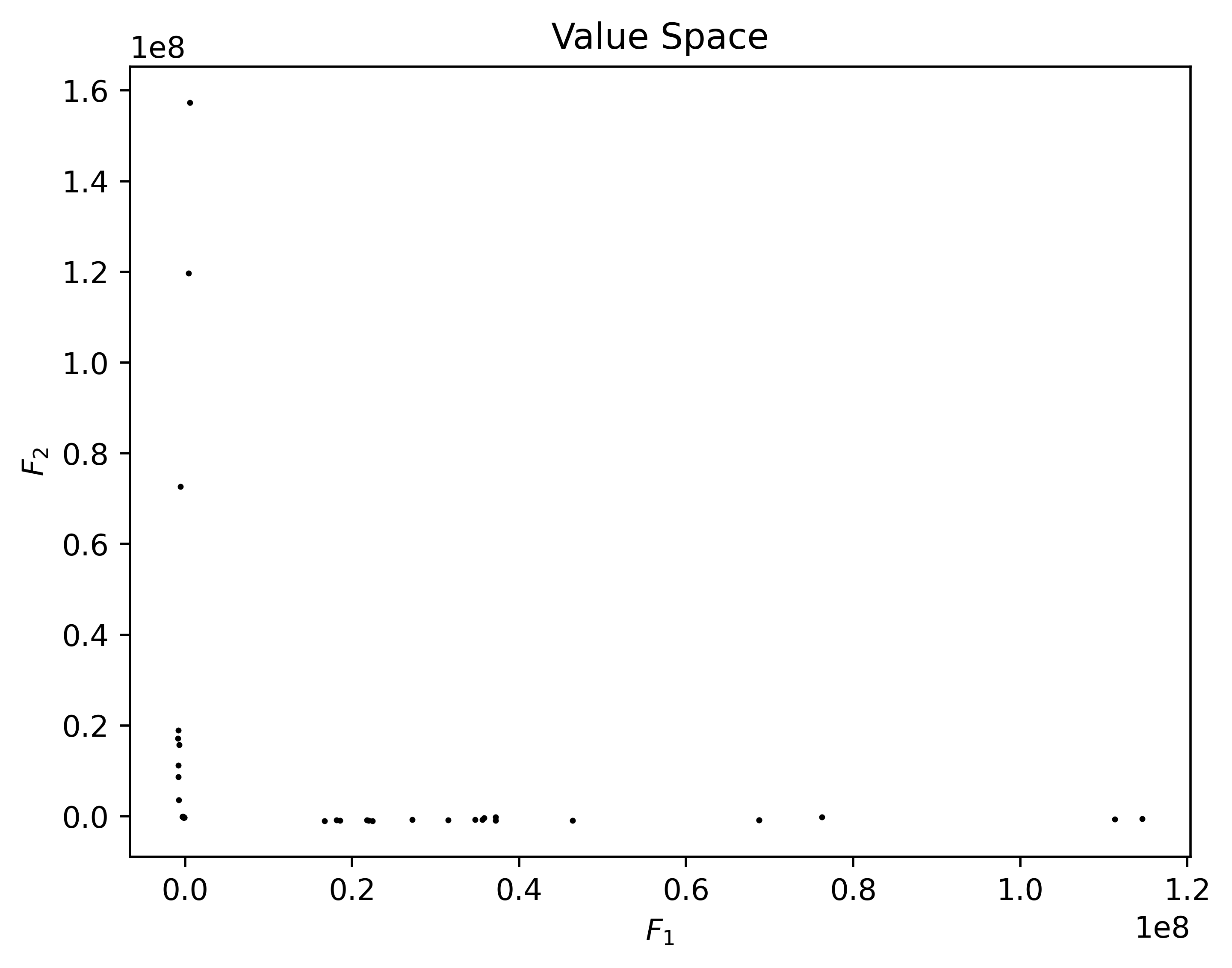} \\
			\includegraphics[scale=0.22]{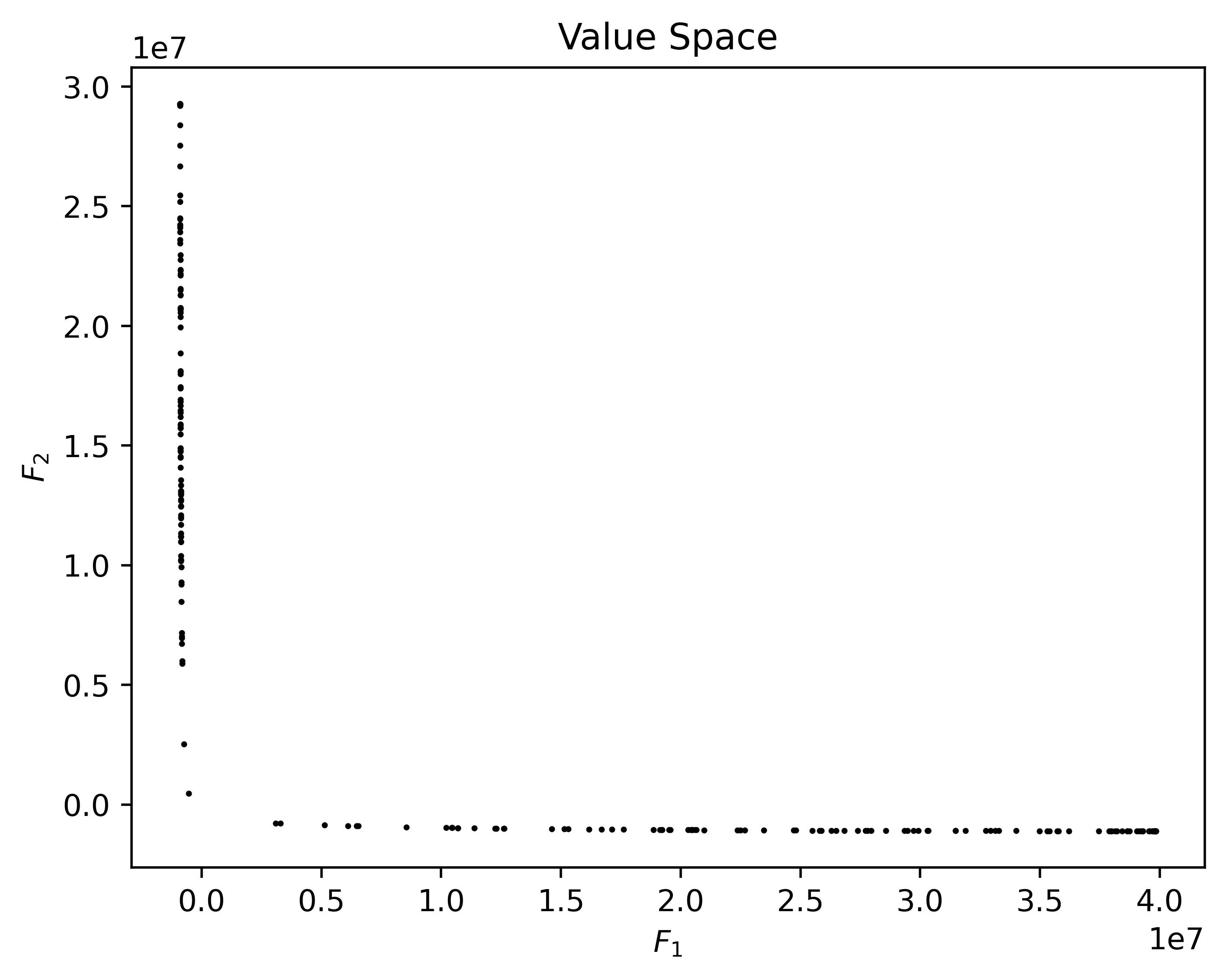} \\
			\includegraphics[scale=0.22]{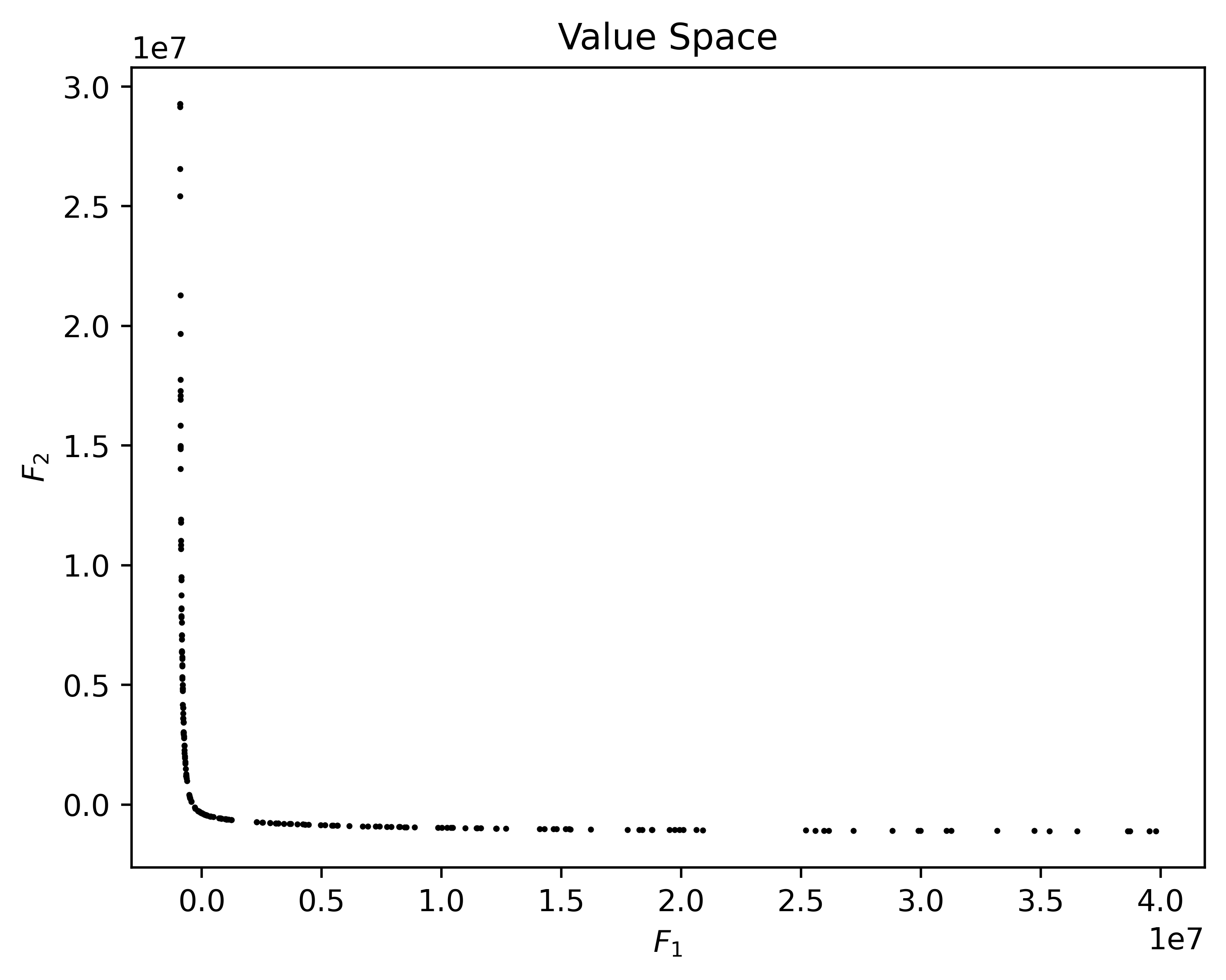}
		\end{minipage}
	}
	\subfigure[QPh]
	{
		\begin{minipage}[H]{.22\linewidth}
			\centering
			\includegraphics[scale=0.22]{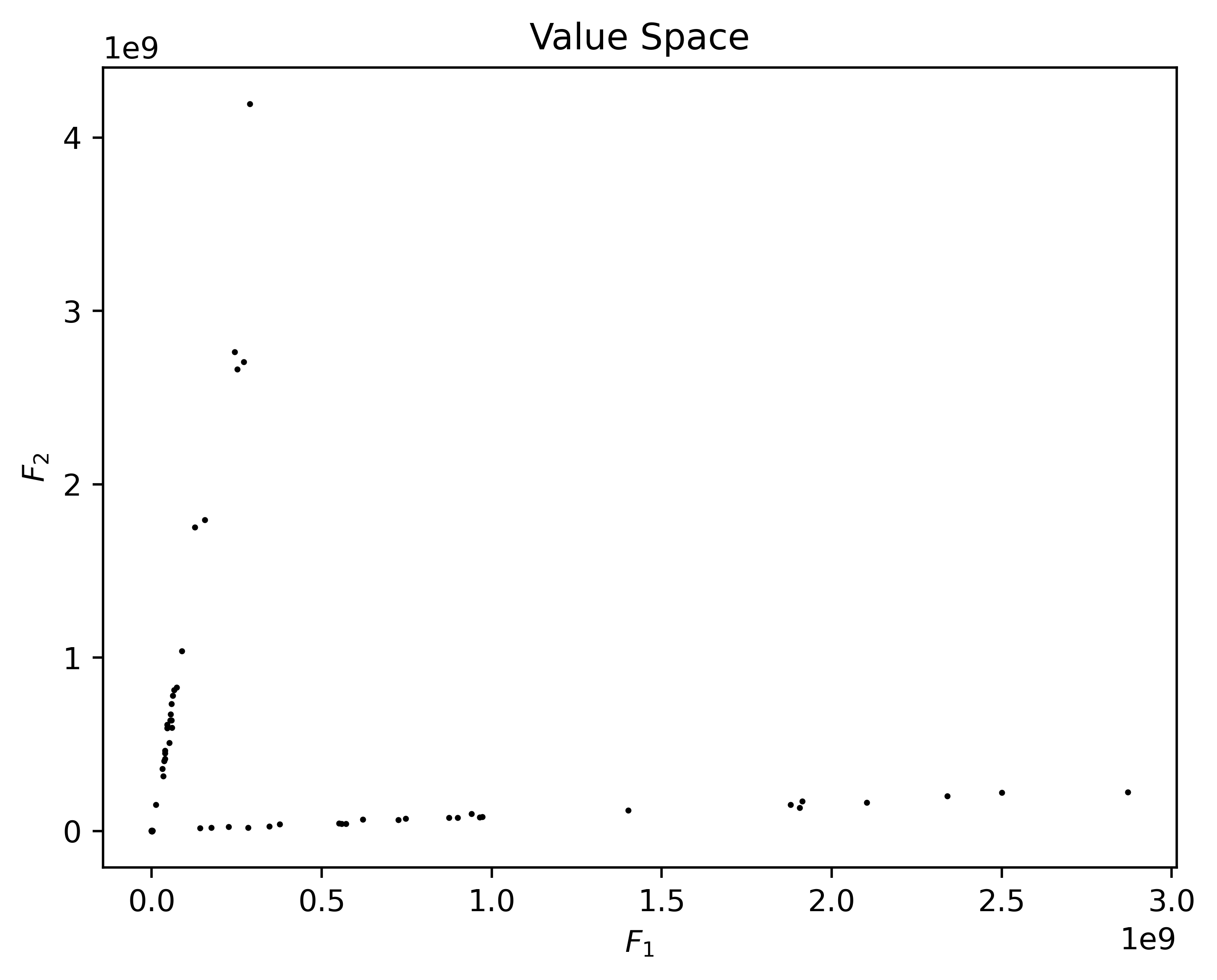} \\
			\includegraphics[scale=0.22]{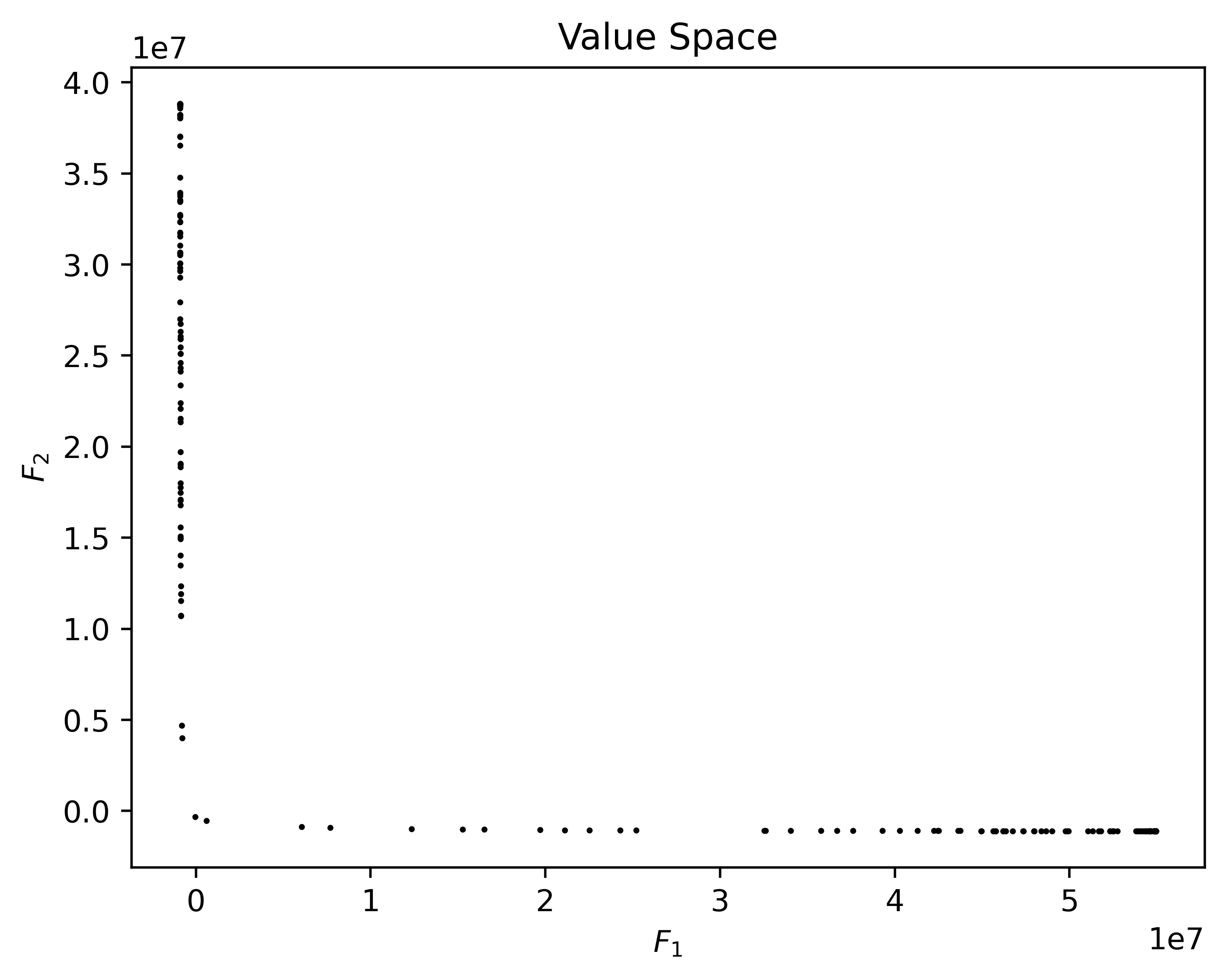} \\
			\includegraphics[scale=0.22]{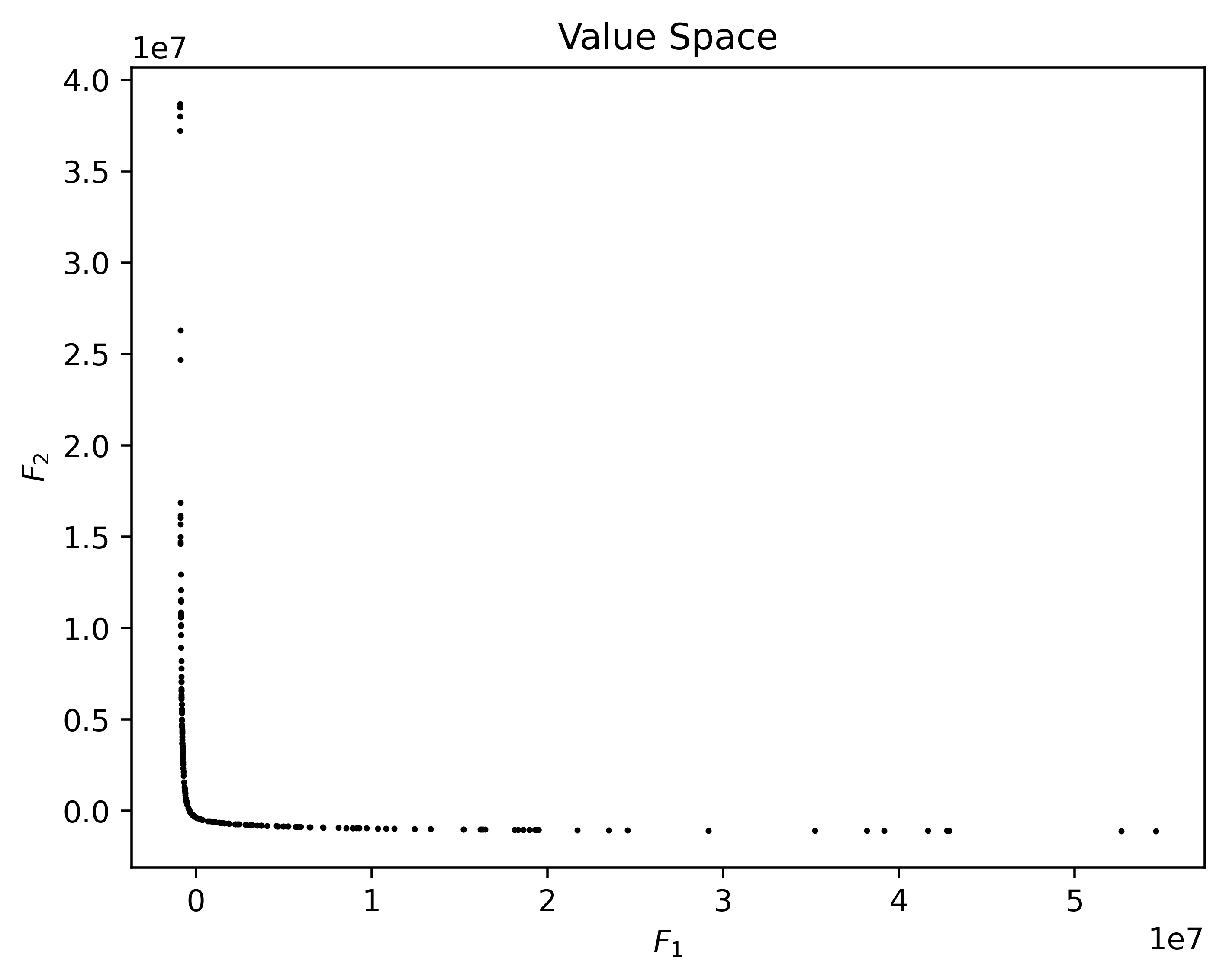}
		\end{minipage}
	}
	\caption{Numerical results in value space obtained by BBDMO ({\bf top}), BBQNMO ({\bf middle}) and SMBBMO for problems QPe, QPf, QPg, and QPh.}
	\label{f4}
\end{figure}
Table \ref{tab4} illustrates the average number of iterations (iter), average number of function evaluations (feval), and average CPU time (time in milliseconds) obtained from 200 experimental runs for each quadratic problem. BBDMO, being a first-order method, exhibits competence in handling moderately ill-conditioned problems (QPb-e) owing to the Barzilai-Borwein rule, yet it struggles to converge within 500 iterations on extremely ill-conditioned problems (QPf-h). Conversely, for ill-conditioned and high-dimensional problems (QPe-h), SMBBMO demonstrates a notable superiority over BBQNMO in terms of CPU time efficiency. It is notable that SMBBMO shows promise in capturing the local curvature of ill-conditioned problems. To sum up, the primary experimental results underscore that SMBBMO achieves a faster convergence rate than BBDMO while maintaining a lower computational cost than BBQNMO.
\section{Conclusions}\label{sec8}
In this paper, we introduce a novel subspace minimization Barzilai-Borwein method for MOPs, which outperforms BBDMO in terms of convergence rate while requiring lower computational resources compared to BBQNMO. We employ a modified Cholesky factorization to ensure global convergence of the proposed method in non-convex scenarios. Our numerical experiments demonstrate that SMBBMO exhibits promising performance for tackling large-scale and ill-conditioned MOPs.
\par From a methodological perspective, it may be worth considering the following points:
\begin{itemize}
	\item By selecting different subspaces, more historical iteration information (see \cite{A2014, CYZ2023m}) can be utilized to construct the subspace. 
	\item By selecting different approximate models, SMCG with cubic regularization \cite{ZLL2021} can also be extended to MOPs.
	 
\end{itemize}

\begin{acknowledgements}
	This work was funded by the Major Program of the National Natural Science Foundation of China [grant numbers 11991020, 11991024]; the National Natural Science Foundation of China [grant numbers 11971084, 12171060]; NSFC-RGC (Hong Kong) Joint Research Program [grant number 12261160365]; the Team Project of Innovation Leading Talent in Chongqing [grant number CQYC20210309536]; the Natural Science Foundation of Chongqing [grant number ncamc2022-msxm01]; Major Project of Science and Technology Research Rrogram of Chongqing Education Commission of China [grant number KJZD-M202300504]; and Foundation of Chongqing Normal University [grant numbers 22XLB005, 22XLB006].
\end{acknowledgements}

\bibliographystyle{abbrv}
\bibliography{references}

\begin{thebibliography}{10}

\bibitem{A2014}
N.~Andrei.
\newblock An accelerated subspace minimization three-term conjugate gradient
  algorithm for unconstrained optimization.
\newblock {\em Numerical Algorithms}, 65:859--874, 2014.

\bibitem{AP2021}
M.~A.~T. Ansary and G.~Panda.
\newblock A globally convergent {SQCQP} method for multiobjective optimization
  problems.
\newblock {\em SIAM Journal on Optimization}, 31(1):91--113, 2021.

\bibitem{BI2005}
H.~Bonnel, A.~N. Iusem, and B.~F. Svaiter.
\newblock Proximal methods in vector optimization.
\newblock {\em SIAM Journal on Optimization}, 15(4):953--970, 2005.

\bibitem{CL2016}
G.~A. Carrizo, P.~A. Lotito, and M.~C. Maciel.
\newblock Trust region globalization strategy for the nonconvex unconstrained
  multiobjective optimization problem.
\newblock {\em Mathematical Programming}, 159(1):339--369, 2016.

\bibitem{CTY2023a}
J.~Chen, L.~P. Tang, and X.~M. Yang.
\newblock A {B}arzilai-{B}orwein descent method for multiobjective optimization
  problems.
\newblock {\em European Journal of Operational Research}, 311(1):196--209,
  2023.

\bibitem{CTY2023b}
J.~Chen, L.~P. Tang, and X.~M. Yang.
\newblock Barzilai-{B}orwein proximal gradient methods for multiobjective
  composite optimization problems with improved linear convergence.
\newblock {\em arXiv preprint arXiv:2306.09797v2}, 2023.

\bibitem{CTY2024p}
J.~Chen, L.~P. Tang, and X.~M. Yang.
\newblock Preconditioned {B}arzilai-{B}orwein methods for multiobjective
  optimization problems.
\newblock {\em optimization-online}, 2024.

\bibitem{CYZ2023m}
W.~Chen, X.~M. Yang, and Y.~Zhao.
\newblock Memory gradient method for multiobjective optimization.
\newblock {\em Applied Mathematics and Computation}, 443:127791, 2023.

\bibitem{DK2016}
Y.~H. Dai and C.~X. Kou.
\newblock A {B}arzilai-{B}orwein conjugate gradient method.
\newblock {\em Science China Mathematics}, 59:1511--1524, 2016.

\bibitem{DD1998}
I.~Das and J.~E. Dennis.
\newblock Normal-boundary intersection: {A} new method for generating the
  {P}areto surface in nonlinear multicriteria optimization problems.
\newblock {\em SIAM Journal on Optimization}, 8(3):631--657, 1998.

\bibitem{D1999}
K.~Deb.
\newblock Multi-objective genetic algorithms: {P}roblem difficulties and
  construction of test problems.
\newblock {\em Evolutionary Computation}, 7(3):205--230, 1999.

\bibitem{EE2024}
Y.~Elboulqe and M.~El~Maghri.
\newblock An explicit spectral {F}letcher--{R}eeves conjugate gradient method
  for bi-criteria optimization.
\newblock {\em IMA Journal of Numerical Analysis}, page drae003, 2024.

\bibitem{E1984}
G.~Evans.
\newblock Overview of techniques for solving multiobjective mathematical
  programs.
\newblock {\em Management Science}, 30(11):1268--1282, 1984.

\bibitem{FD2009}
J.~Fliege, L.~M. Gra{\~n}a~Drummond, and B.~F. Svaiter.
\newblock Newton's method for multiobjective optimization.
\newblock {\em SIAM Journal on Optimization}, 20(2):602--626, 2009.

\bibitem{FS2000}
J.~Fliege and B.~F. Svaiter.
\newblock Steepest descent methods for multicriteria optimization.
\newblock {\em Mathematical Methods of Operations Research}, 51(3):479--494,
  2000.

\bibitem{FV2016}
J.~Fliege and A.~I.~F. Vaz.
\newblock A method for constrained multiobjective optimization based on {SQP}
  techniques.
\newblock {\em SIAM Journal on Optimization}, 26(4):2091--2119, 2016.

\bibitem{FW2014}
J.~Fliege and R.~Werner.
\newblock Robust multiobjective optimization \& applications in portfolio
  optimization.
\newblock {\em European Journal of Operational Research}, 234(2):422--433,
  2014.

\bibitem{FGM2022}
E.~H. Fukuda, L.~M. Gra{\~n}a~Drummond, and A.~M. Masuda.
\newblock A conjugate directions-type procedure for quadratic multiobjective
  optimization.
\newblock {\em Optimization}, 71(2):419--437, 2022.

\bibitem{GLP2022}
M.~L.~N. Gon{\c{c}}alves, F.~S. Lima, and L.~F. Prudente.
\newblock Globally convergent {N}ewton-type methods for multiobjective
  optimization.
\newblock {\em Computational Optimization and Applications}, 83(2):403--434,
  2022.

\bibitem{GP2020}
M.~L.~N. Gon{\c{c}}alves and L.~F. Prudente.
\newblock On the extension of the {H}ager--{Z}hang conjugate gradient method
  for vector optimization.
\newblock {\em Computational Optimization and Applications}, 76(3):889--916,
  2020.

\bibitem{GI2004}
L.~M. Gra{\~n}a~Drummond and A.~N. Iusem.
\newblock A projected gradient method for vector optimization problems.
\newblock {\em Computational Optimization and Applications}, 28(1):5--29, 2004.

\bibitem{HCL2023}
Q.~R. He, C.~R. Chen, and S.~J. Li.
\newblock Spectral conjugate gradient methods for vector optimization problems.
\newblock {\em Computational Optimization and Applications}, 86(2):457--489,
  2023.

\bibitem{Hil1}
C.~Hillermeier.
\newblock Generalized homotopy approach to multiobjective optimization.
\newblock {\em Journal of Optimization Theory and Applications},
  110(3):557--583, 2001.

\bibitem{HZC2024}
Q.~J. Hu, L.~P. Zhu, and Y.~Chen.
\newblock Alternative extension of the {H}ager--{Z}hang conjugate gradient
  method for vector optimization.
\newblock {\em Computational Optimization and Applications}, pages 1--34, 2024.

\bibitem{BK1}
S.~Huband, P.~Hingston, L.~Barone, and L.~While.
\newblock A review of multiobjective test problems and a scalable test problem
  toolkit.
\newblock {\em IEEE Transactions on Evolutionary Computation}, 10(5):477--506,
  2006.

\bibitem{LLL2024}
M.~Lapucci, G.~Liuzzi, S.~Lucidi, and M.~Sciandrone.
\newblock A globally convergent gradient method with momentum.
\newblock {\em arXiv preprint arXiv.2403.17613}, 2024.

\bibitem{LM2023}
M.~Lapucci and P.~Mansueto.
\newblock A limited memory quasi-{N}ewton approach for multi-objective
  optimization.
\newblock {\em Computational Optimization and Applications}, 85(1):33--73,
  2023.

\bibitem{LP2018}
L.~R. Lucambio~P\'{e}rez and L.~F. Prudente.
\newblock Nonlinear conjugate gradient methods for vector optimization.
\newblock {\em SIAM Journal on Optimization}, 28(3):2690--2720, 2018.

\bibitem{MA2004}
R.~T. Marler and J.~S. Arora.
\newblock Survey of multi-objective optimization methods for engineering.
\newblock {\em Structural and Multidisciplinary Optimization}, 26(6):369--395,
  2004.

\bibitem{MP2019}
V.~Morovati and L.~Pourkarimi.
\newblock Extension of {Z}outendijk method for solving constrained
  multiobjective optimization problems.
\newblock {\em European Journal of Operational Research}, 273(1):44--57, 2019.

\bibitem{M1980}
H.~Mukai.
\newblock Algorithms for multicriterion optimization.
\newblock {\em IEEE Transactions on Automatic Control}, 25(2):177--186, 1980.

\bibitem{P2014}
{\v{Z}}.~Povalej.
\newblock Quasi-{N}ewton's method for multiobjective optimization.
\newblock {\em Journal of Computational and Applied Mathematics}, 255:765--777,
  2014.

\bibitem{PN2006}
M.~Preuss, B.~Naujoks, and G.~Rudolph.
\newblock Pareto set and {EMOA} behavior for simple multimodal multiobjective
  functions", booktitle="parallel problem solving from nature - ppsn ix.
\newblock pages 513--522, Berlin, Heidelberg, 2006. Springer Berlin Heidelberg.

\bibitem{QG2011}
S.~J. Qu, M.~Goh, and F.~T. Chan.
\newblock Quasi-{N}ewton methods for solving multiobjective optimization.
\newblock {\em Operations Research Letters}, 39(5):397--399, 2011.

\bibitem{SK2018}
O.~Sener and V.~Koltun.
\newblock Multi-task learning as multi-objective optimization.
\newblock In S.~Bengio, H.~Wallach, H.~Larochelle, K.~Grauman, N.~Cesa-Bianchi,
  and R.~Garnett, editors, {\em Advances in Neural Information Processing
  Systems}, volume~31. Curran Associates, Inc., 2018.

\bibitem{TFY2023n}
H.~Tanabe, E.~H. Fukuda, and N.~Yamashita.
\newblock New merit functions for multiobjective optimization and their
  properties.
\newblock {\em Optimization}, pages 1--38, 2023.

\bibitem{W2012}
K.~Witting.
\newblock {\em Numerical algorithms for the treatment of parametric
  multiobjective optimization problems and applications}.
\newblock PhD thesis, Paderborn, Universit{\"a}t Paderborn, Diss., 2012, 2012.

\bibitem{YS1995}
Y.~X. Yuan and J.~Stoer.
\newblock A subspace study on conjugate gradient algorithms.
\newblock {\em Zeitschrift f{\"u}r Angewandte Mathematik und Mechanik},
  75(1):69--77, 1995.

\bibitem{ZLL2021}
T.~Zhao, H.~W. Liu, and Z.~X. Liu.
\newblock New subspace minimization conjugate gradient methods based on
  regularization model for unconstrained optimization.
\newblock {\em Numerical Algorithms}, 87:1501--1534, 2021.

\end{thebibliography}

\end{document}